\documentclass[a4paper,12pt]{amsart} 

\usepackage{soul}
\usepackage{amsmath, amssymb, amsfonts, amsthm}
\usepackage{mathrsfs}
\usepackage{enumerate}
\usepackage{stmaryrd}
\usepackage[utf8]{inputenc} 
\usepackage[T1]{fontenc}    
\usepackage{color}
\usepackage{graphicx}
\usepackage{bm}
\usepackage{hyperref}
\usepackage{wrapfig}
\usepackage{cleveref}

\theoremstyle{plain}
\numberwithin{equation}{section}
\newtheorem{theorem}{Theorem}[section]

\newtheorem{proposition}[theorem]{Proposition}
\newtheorem{lemma}[theorem]{Lemma}

\newtheorem{corollary}[theorem]{Corollary}
\theoremstyle{remark}
\newtheorem{remark}[theorem]{Remark}

\usepackage[left=2.5cm,right=2.5cm,vmargin=2.5cm]{geometry}
\usepackage{mathrsfs, stmaryrd, mathtools}

\DeclarePairedDelimiterX\intff[2]{[}{]}{#1,#2}
\DeclarePairedDelimiterX\intfo[2]{[}{)}{#1,#2}
\DeclarePairedDelimiterX\intof[2]{(}{]}{#1,#2}
\DeclarePairedDelimiterX\intoo[2]{(}{)}{#1,#2}
\DeclarePairedDelimiter{\pars}{(}{)}

\DeclarePairedDelimiter{\absolute}{|}{|}

\DeclarePairedDelimiterX{\setof}[2]{\lbrace}{\rbrace}{#1\,{\colon}\,#2}
\DeclarePairedDelimiterX{\bracksof}[2]{[}{]}{#1\,\delimsize|\,#2}
\DeclarePairedDelimiterX{\parsof}[2]{(}{)}{#1\,\delimsize|\,#2}
\DeclarePairedDelimiterXPP\lnorm[2]{}\lVert\rVert{_{#1}}{#2}

\def\n{\mathbb N}
\def\z{\mathbb Z}
\def\r{\mathbb R}
\def\p{\mathbb P} %%Law wrt everything%%
\def\P{\mathbf P} %%law of BRW%%
\def\E{\mathbf E} %% expectation wrt BRW%%  
\def\RR{\mathfrak R}
\def\BScap{\mbox{\rm BScap}}%%Snake capacity%%
\def\Bcap{\mbox{\rm Bcap}}%%Discrete snake capacity
\def\ball{{\rm B}}  %%Ball in Rd or Zd%%
\def\brwrange{{\mathfrak R}}
\renewcommand{\Cap}{\mathrm{Cap}}
\newcommand{\rd}{\mathrm{d}}
\newcommand{\ind}{\mathbb{1}}
\newcommand{\cA}{\mathcal{A}}

\newcommand{\cK}{\mathcal{K}}

\newcommand{\cO}{\mathcal{O}} %%Open set%%
\newcommand{\open}{D}%\mathrm{O}} %%Open set%%
\newcommand{\cP}{\mathcal{P}}

\newcommand{\tauB}{T}
\newcommand{\cnbd}[2]{{{#1}^{#2}}}

\renewcommand{\ind}{\mathbf 1}
\newcommand{\dmin}{d_{\min}}
\newcommand{\phit}[1]{{\bf p}_K^{(#1)}} %%Escape probability \mathbf P_x(V_\alpha\cap K\ne\emptyset)
\def\adj{{\mathrm{adj}}}
\def\brwrange{{\mathscr R}}
\def\brwm{M_\theta} %%Transition matrix from general BRW to Bronwian snake%%
\def\coeffm{M_*}%%Transition matrix, version for the conditioned coupling Janson-Marckert%%
\def\pp{\mathbf P} \def\ee{\mathbf E} %law of the theta-random walk. determine if we use the same notation as law of the BRW or not.%
\newcommand{\ttree}[1]{\mathcal T_{#1}} %% Tree models %%
\newcommand{\tbrw}[2]{V_{#1}} %% BRW models %%
\def\I{\rm I} %% subscript for the model T_- \cup T_adj

\newcommand{\hdist}{d_H}

\def\diam{{\rm diam}} %diameter of K%%
\def\N{\mathcal N} %%Excursion measure for ISE%%%
\usepackage[draft]{fixme} 
\FXRegisterAuthor{jf}{ajf}{JFD}
\fxusetheme{colorsig}

\title{Branching capacity and Brownian snake capacity}

\date{\today}
\author{Tianyi Bai}
\address{Tianyi Bai, Chinese Academy of Sciences,   China.}
\email{tianyi.bai73@amss.ac.cn}

\author{Jean-Fran\c{c}ois Delmas}
\address{Jean-Fran\c{c}ois Delmas, CERMICS,
Ecole des Ponts,  France}
\email{delmas@cermics.enpc.fr}

\author{Yueyun Hu}
\address{Yueyun Hu, 
LAGA, Universit\'e Paris XIII, 99 av. J.B. Cl\'ement, 93430 Villetaneuse,
France}
\email{yueyun@math.univ-paris13.fr}

\begin{document}
\begin{abstract}
The branching capacity has been introduced by Zhu \cite{zhu2016critical}  as the limit of the hitting probability of a symmetric branching random walk in $\z^d$, $d\ge 5$. Similarly, we define the Brownian snake capacity  in $\r^d$, as the scaling limit of the hitting 
probability by the Brownian snake starting from afar. Then, we prove our main result on the vague convergence of the  rescaled branching capacity towards this Brownian snake capacity. 
Our proof relies on a precise convergence rate for the approximation of the branching capacity by hitting probabilities. 
\end{abstract}
    


\maketitle

\section{Introduction}
Let  $\ttree  c$  be  a   critical  Galton-Watson  tree  with  offspring
distribution $\mu=(\mu(i))_{i\ge 0}$ with mean 1.   In other words, $\ttree  c$ is a
random  tree   that  begins  with   one  particle,  and   each  particle
independently produces a random number  of offspring according to $\mu$.
It is well-known that $\ttree c$ is
almost surely finite provided 
$\mu(1)< 1$.

Let $\theta$ be a probability distribution on $\z^d$. For $x\in\z^d$, we denote by $\P_x$ the law of  a branching random walk (BRW) $\tbrw c  x$ in $\z^d$,  which is  a  $\theta$-random walk  indexed by  the  tree $\ttree  c$
constructed as follows.  We pin the root  of $\ttree c$ at $x$, then to
each edge  of $\ttree  c$ we attach  an independent  random displacement
distributed  as $\theta$.  The position  of a  vertex of  $\ttree c$  is
defined as  $x$ plus the  sum of  all displacements associated  with the
edges in the simple path from the root to that vertex. The collection of
all (spatial) positions of vertices of $\ttree c$ is called the range of
$\tbrw c  x$ and  is denoted by  $\brwrange_c$. 
 
We consider  $d\ge 5$. For $x\in \r^d$, let $|x|$ denote its  usual Euclidean norm, and $\ball(x,r):= \{y \in \r^d: |y-x| <r\}$ the open ball centered at $x$ with radius $r> 0$. 
 Throughout the article, we shall  assume that 
\begin{align}
& \text{$\mu$ has mean $1$ and variance $\sigma^2\in (0, \infty)$},
\label{hyp-tree} 
\\
\label{hyp-brw} 
& \text{$\theta$ is symmetric, irreducible with covariance matrix $\brwm$
                   and}\\
&   \nonumber         \text{there exists a finite constant $c$ such that for all $r>0$:}  \quad \theta\big(\ball(0, r)^c\big) \le c\, r^{-d}.
\end{align} 
The last condition in~\eqref{hyp-brw}  is in particular satisfied when $\theta$ has  a  finite $d$-th moment.

\medskip

To introduce the branching capacity, we will use the Green function of a
random walk. Let $(S_n)_{n\ge 0}$ be a $\theta$-random walk  on $\z^d$, where $\theta$ serves as the step distribution of $(S_n)$. 
For $x\in \r^d$,  we also
consider its norm in relation with $(S_n)$ given by $
|x|_\theta:=\sqrt{x^T\brwm^{-1} x}$.  We define the Green function of $(S_n)$ by $g(x,y):=g(x-y)$ and for
$x\in \z$
\begin{equation}\label{c_g}
g(x):=\sum_{n=0}^\infty \pp_0(S_n=x) = c_{g}\, | x|_\theta^{2-d} + O(|x|^{1-d}), 
\end{equation}
with 
\[
c_{g}:= \frac{\Gamma(\frac {d-2} 2)}{2\pi^{d/2}\sqrt{\det\brwm}},
\]
where the asymptotic is due to Uchiyama \cite[Theorem 2]{MR1625467}. 
Let $K\subset \z^d$ be a nonempty finite set. Following Zhu
\cite{zhu2016critical}, we define the branching capacity of $K$ as
\begin{equation}
  \label{def-bcap}
  \Bcap(K):= \lim_{x \to\infty} \frac{\P_x(\brwrange_c \cap K\neq \emptyset)}{g(x)}
  =\frac1{c_{g}}\lim_{x \to\infty} | x|_\theta^{d-2}{\P_x(\brwrange_c \cap K\neq \emptyset)},   
\end{equation}
where $x\rightarrow\infty$ means  $|x|\rightarrow \infty $.
Moreover, $\Bcap(\cdot)$ viewed as a set function is non-decreasing, invariant under translations, and strictly positive when the set is not empty.

In our first  result we give 
a  rate of convergence  for \eqref{def-bcap}, which will  be
useful in the study of the branching capacity. 

\begin{theorem}\label{theo:compareBcap}
Let $d\ge 5$ and $\lambda>1$.
Assume \eqref{hyp-tree} and \eqref{hyp-brw}. 
There exists  a positive constant $C=C(d,\lambda)$ such that uniformly in  $r\ge 1$,  $K \subset \ball(0,r)$ nonempty and $x\in \z^d$ with $|x|\ge \lambda r$, we have, with  $\alpha=(d-4)/2(d-1)$,
\[
\Big|\Bcap(K)- \frac{ \P_x(\brwrange_c\cap K\neq \emptyset)}{g(x)} \Big| \le C   \left(\frac{r}{|x|}\right)^\alpha \, \Bcap(K).  
\]
\end{theorem}

We do not claim that the value of $\alpha$ is optimal. 
We refer to Proposition \ref{c:compareBcap} and Proposition \ref{c:compareBcap2} for corresponding results on the branching models $\tbrw{\text{adj}}{}$, $\tbrw I{}$ and $\tbrw - {}$ defined in Section \ref{sec:tadj}.

Branching capacity is a relatively new concept defined and explored in a
series of  works by Zhu \cite{zhu2016critical,  zhu-2, zhu2021critical}.
In  particular, Zhu  \cite{zhu2016critical} showed  that $\Bcap(K)$  can
also be defined through the exit  probabilities from a BRW indexed by an
infinite  tree,   see  \eqref{Bcap-eK}.  Therefore,  Le   Gall  and  Lin
\cite[Proposition   3]{LeGall-Lin-range}'s   ergodicity   implies   that
$\Bcap(K)$ is the  almost sure limit of the rescaled  cardinality of the
sum of  $K$ and the  first $n$ points  in this BRW.   Recently, Asselah,
Schapira,  and Sousi  \cite{asselah2023local} studied  the link  between
branching capacity and Green's function.  In addition, when $(S_n)$ is a
simple  random  walk on  $\z^d$,  Asselah,  Okada, Schapira,  and  Sousi
\cite{AOSS}  demonstrated  that  $\Bcap(K)$  can be  compared  with  the
rescaled cardinality  of the  sum of  $K$ and the  Minkowski sum  of the
ranges of two  independent copies of $(S_n)$.  We also  mention a recent
study by Schapira \cite{Schapira}  and an ongoing work \cite{bdh-5d}
on the branching capacity of the range
of a random walk.

It is  a natural problem to  investigate the scaling limit  of branching
capacities. Loosely  speaking, the  branching capacity  is related  to a
critical  BRW  in  the  same  way as  the  (discrete)
Newtonian capacity does to a random  walk.  Therefore, we first define a
continuous counterpart  of $\Bcap$,  called Brownian snake  capacity, by
utilizing the  Brownian snake, which is the scaling limit of BRW.  The  introduction of the  Brownian snake
capacity aims to  establish a connection between  the branching capacity
and the  hitting probabilities  associated with  the Brownian  snake, as
investigated  by   Le  Gall  \cite{LeGall1994},  Dhersin   and  Le  Gall
\cite{dlg1997}, and Delmas \cite{Delmas99}.

Following Le  Gall \cite{LeGall1999}, let  $\mathcal N_x(\rd W)$  be the
excursion measure of a Brownian snake $W=(W_t)_{t\ge 0}$ started from $x\in
\r^d$, 
see Section~\ref{sec:BrwonianSnake} for  the precise definitions. Denote
by   $\mathfrak   R$   the   range    of   the   Brownian   snake   (see
\eqref{def-rangeBrownianSnake}).   In  the   next  theorem,   we  define
$\BScap(A)$ the Brownian snake capacity of $A$ for $d\geq 5$, and refer to 
Remark \ref{rmk:low_dim} for the case $d\leq 4$. 
\begin{theorem} \label{p:snakecap} 
Let $d\ge 5$. For any bounded Borel set $A \subset \r^d$, the following limit exists and is finite:
\[   \BScap(A):=\lim_{x\to\infty} |x|^{d-2}  \N_x(\RR\cap A\neq \emptyset). 
\] 
\end{theorem}

In the next  proposition we state some properties of  the Brownian snake
capacity.   Choquet   capacities  are   precisely  defined   in  Section
\ref{sec:capacity}, see also~\eqref{eq:sub+} for the strong sub-additive
property; the Riesz capacity $\Cap_{d-\alpha}$ for $\alpha\in (0, d)$ is
defined  in  Remark \ref{rem:cap_d=cap}  below;  the  closure of  a  set
$A\subset\r^d$ is denoted by $\overline A$.

\begin{proposition}\label{prop:propertiesBScap}
  Let $d\ge 5$.
\begin{enumerate}
\item\label{it:BS=cap}
  \textbf{Choquet capacity.} The map $\BScap$ can be extended into a strongly sub-additive Choquet
capacity on $\cP(\r ^d)$,   the power set of $\r^d$. 

\item\label{it:BS=scaling}
  \textbf{Scaling and translation.}
For any $x\in\r^d$, $a\in \r$ and $A\subset \r^d$, we have
\begin{equation}\label{eq:BScap_scaling}
\BScap(x+aA)=a^{d-4}\, \BScap(A).
\end{equation}
\item\label{it:BS=d-4}
  \textbf{Comparison with  $\Cap_{d-4}$.}
There exist finite positive constants $c_1$ and $c_2$ such that for any $A\subset \r^d$
\begin{equation}\label{eq:BScap_Cap_d-4}
c_1\Cap_{d-4}(A)\le \BScap(A)\le c_2\Cap_{d-4}(A).
\end{equation}
\item \label{it:BS=reg}
  \textbf{Regularity.}  Let $\open\subset\r^d$ be a  bounded open set. If for all $y\in
\partial_{\text{ext}}\open$,  that is,  the boundary of the unbounded connected component of $D^c$,
\begin{align}\label{eq:d-2_condition}
\liminf_{n\rightarrow \infty } 2^{n(d-2)} \, \Cap_{d-2} \left(\open  \cap
   \ball(y, 2^{-n}) \right)>0,
\end{align}
then, we have
\[
\BScap(\open)=\BScap(\overline{\open}).
\]

\item\label{it:BS=6}
  \textbf{Particular value.} For $d=6$, we have
  $\BScap(\ball(0,1))=6$. 
\end{enumerate}
\end{proposition}

The  scaling and  translation properties  of the  capacity $\BScap$  are
consequences  of  its  definition  and the  scaling  properties  of  the
Brownian  snake; the  comparison  with the  Riesz capacity  $\Cap_{d-4}$
comes  from  estimates  in~\cite{dlg1997},  but it  is  unclear  whether
$\BScap$ and $\Cap_{d-4}$  are equal (up to  a multiplicative constant);
the regularity of $\BScap$ on open sets relies on uniqueness of solution
of  $\Delta u=  4u^2$ on  open  set with  infinite boundary  conditions,
see~\cite{MR1341844} (this has  to be compared with the  Wiener test for
Brownian                 motion                 given                 by
$\sum_{n=0}^\infty  2^{n(d-2)}  \Cap_{d-2}   \left(\open  \cap  \ball(y,
  2^{-n})  \right)=\infty  $),   as  mentioned  in~\cite{MR1341844}  the
condition~\eqref{eq:d-2_condition} is  satisfied when  $D$ has  a Lipschitz
boundary; we are able to compute  the capacity of the balls in dimension
6 by  giving an explicit  formula for the  solution of $\Delta  u= 4u^2$
outside the unit  ball $\ball(0, 1)$ with  infinite boundary conditions,
see Remark~\ref{rem:u1} and the new formula~\eqref{eq:u-for-d6} therein.

\medskip

We refer to \cite{asselah2023local} for the analogue of
\eqref{eq:BScap_Cap_d-4} for branching capacity on finite subsets of $\z^d$ in the case when
$\theta$ has bounded jumps.
To state our next result we extend the capacity $\Bcap$ on $\z^d$ as a Choquet capacity on
$\r^d$ by setting $\Bcap(A)=\Bcap(A \cap \z^d)$ for all $A\subset \r^d$. 
For every set  $A\subset\r^d$ and $\varepsilon>0$, denote 
by  $\cnbd{A}{\varepsilon}=\setof*{x}{d(x,A)\le\varepsilon}$ the  closed
$\varepsilon$-neighborhood of $A$,  where $d(x,A):=\inf_{y\in  A}  |x-y|$.

On exploring the  link between  the discrete  branching capacity  and the  continuous
Brownian snake capacity, we have the following result.
Let $\cK$ and $\cO$ denote the families of compact and open subsets of
$\r^d$. 
\begin{theorem}\label{prop:A^epsilon}
  Let $d\ge 5$.
  Assume \eqref{hyp-tree} and \eqref{hyp-brw}.  
  We have
\begin{align} 
  \label{upper:nK}
  \limsup_{n\rightarrow\infty}\frac{\Bcap(nK)}{n^{d-4}} \le
  c_\theta\,\BScap(\brwm^{-1/2}K) \quad\text{for}\quad
  K\in \cK, 
  \\
  \label{lower:Open}
\liminf_{n\rightarrow\infty}\frac{\Bcap(n\open)}{n^{d-4}}\ge
  c_\theta\,\BScap(\brwm^{-1/2}\open)
  \quad\text{for}\quad
  \open\in \cO, 
\end{align}
with
\[
c_\theta:= \frac {2}
{\sigma^2{c_{g}}}= \frac {4\pi^{d/2}\, \sqrt{\det\brwm}}
{\sigma^2\, {\Gamma(\frac{d-2}{2})}}\cdot
\]
\end{theorem}

In the setting of Norberg \cite{norberg},
Equations~\eqref{upper:nK}-\eqref{lower:Open} correspond to the vague
convergence of  the capacities
$n^{-(d-4)} \Bcap(n \, \cdot\, )$ towards  the capacity
$c_\theta\BScap(\brwm^{-1/2}\, \cdot\, )$.

\begin{remark}
If $\theta$ is the uniform probability  distribution on  the $2d$ unit
vectors in $\z^d$, then $\brwm=\frac 1 d I$, and for any set $A\subset \r^d$ we have
\[
  c_\theta\, \BScap(\brwm^{-1/2}    A    )   =    \frac{2\pi^{d/2}}{\sigma^2
    \Gamma(\frac{d}{2}) }\, \BScap(A).
\]
\end{remark}

The paper  is organized as follows.   In Section \ref{sec:proof-BS-cap},
we   explore   the  Brownian   snake,   and   provide  proofs  for   Theorem
\ref{p:snakecap} and Proposition  \ref{prop:propertiesBScap}. In Section
\ref{sec:convergence4},   we  prove   Theorem  \ref{prop:A^epsilon}  by assuming
Theorem \ref{theo:compareBcap}.  And  finally we complete the
proof of Theorem \ref{theo:compareBcap} in Section~\ref{s:appendix}.

\medskip

For notational brevity, let $f_\delta(x),g_\delta(x)$ be two nonnegative functions
depending on some  parameter $\delta$, we write  $f_\delta(x)\lesssim g_\delta(x)$ when
there  is a  constant  $C>0$,  independent of  $x$  and  $\delta$, such  that
$f_\delta(x)\le  Cg_\delta(x)$.   We  also   denote  $f_\delta(x)\asymp   g_\delta(x)$  when
$f_\delta(x)\lesssim g_\delta(x)$ and $g_\delta(x)\lesssim
f_\delta(x)$.

We  write $\n=\z\cap [0, \infty )$ the set of non-negative integers.  We
denote  by $\diam(A):=\sup_{x,y\in A} |x-y|$  the diameter of $A\subset \r^d$.

\section{The Branching and the Brownian snake  capacities}
\label{sec:proof-BS-cap}

\subsection{Choquet capacity}
\label{sec:capacity}

 Let $(X,\cO)$  be an Hausdorff topological space,
$\mathcal  P(X)$ be  its power  set.  Let  $\cA\subset \cP(X)$  be 
stable  by   finite  union   and  intersection.  We   say  that   a  map
$I:\cA\rightarrow[0, +\infty ]$ is non-decreasing if:
\[
  I(A)\leq  I(B)
  \quad\text{for}\quad
  A, B\in \cA
  \quad\text{such that}\quad
  A  \subset B,
\]
and is strongly sub-additive if
\begin{equation}
  \label{eq:sub+}
  I(A\cup B) + I(A\cap B) \leq  I(A)+I(B)
  \quad\text{for}\quad
  A, B\in \cA. 
\end{equation}

Let  $\cK\subset\mathcal  P(X)$ denote the
set of  compact sets (which   is  stable by finite  union and
intersection).
Following                  Definition~27
in~\cite[Section~III.2]{dellacherie-meyer},             a            map
$I:\cP(X)\rightarrow[0, +\infty  ]$ is  a Choquet capacity  (relative to
$\cK$) if:
\begin{enumerate}[(i)]   
\item $I$ is non-decreasing. 
\item If $(A_n)_{n\in \n}$ is a non-decreasing sequence of subsets of $X$,  then 
\[
  I(\cup_{n} A_n)= \sup_{n} I(A_n).
\]
\item \label{it:cap-Kdec}
  If $(K_n)_{n\in \n}$ is a non-increasing sequence in $\cK$,
  then
\[
  I(\cap_{n} K_n)= \inf_{n} I(K_n).
\]
\end{enumerate}
A set $A\subset X$ is capacitable (with respect to $I$) if:
\begin{equation}
   \label{eq:capacitable}
  I(A)=\sup_{K\in \cK, \, K\subset A} I(K). 
\end{equation}
 
Assume that $(X,\cO)$ is a second countable Hausdorff locally compact
space. Then according to~\cite[Theorem~III.13]{dellacherie-meyer} and the comment
above  therein, we get that Borel
sets are  analytic and thus capacitable by Choquet theorem (see also
Theorem~28 therein).

\medskip

Given a map $J:\cK\rightarrow[0,\infty]$ defined only  on the compact sets, we
can  define  a  function $J^*:\mathcal  P(X)\rightarrow  [0,\infty]$  as
follows.  For every open set $\open\in \cO$, we set
\begin{equation}\label{eq:choquet}
  J^*(\open)=\sup_{K\in \cK, \,  K\subset \open} J(K),
\end{equation}
and for any $A\subset X$,
\begin{equation} \label{eq:choquet-borel}
   J^*(A)= \inf_{\open \in \cO , \,  A \subset \open}
   J^*(\open).
\end{equation}

  \begin{remark}
  \label{rem:sub-+J}
  It is not difficult, using \cite[Lemme~p.100]{dellacherie-meyer}, 
  to check that if $J$ is strongly sub-additive on
  $\cK$, then $J^*$ is also strongly sub-additive on
  $\cP(X)$. 
\end{remark}

The next theorem from \cite{dellacherie-meyer} states that $J^*$ extends $J$ into a Choquet capacity. We say that the map $J$ is
right continuous if  for every  $K\in \cK$
and every real  number
$a>  J(K)$,  there  exists  an  open  set  $\open$  such  that
$K\subset \open$ and  $ J(K')<a$ for every compact set $K'\subset
\open$.

\begin{theorem}[{\cite[Theorem~III.42]{dellacherie-meyer}}]
\label{theo:ext-J}
Let    $J:\cK\rightarrow[0,\infty]$    be    non-decreasing,    strongly
sub-additive  and  right continuous. 
Then the map $J^*$ is a Choquet  capacity (relative to $\cK$), and it coincides
with $J$ on $\cK$.  
\end{theorem}

\begin{remark}
  On $X=\z^d$ with  the Euclidean distance, $\cK$ simply  consists of finite
  sets.  Hence  one can  easily  check  that  $\Bcap$ defined  on  $\cK$
  by~\eqref{def-bcap}    satisfies    the    conditions    of    Theorem
  \ref{theo:ext-J}. Hence we can extend  it to a Choquet capacity, which
  we still denote by $\Bcap$. Notice that $\Bcap$ is strictly positive
  (except for empty set) thanks to~\eqref{bcap>0}.
\end{remark}

\begin{remark}
\label{rem:cap_d=cap}
Let  $X=\r^d$ and  $\gamma\in  (0,  d)$. We  consider  the Riesz  kernel
$k_\gamma(x)=          C_{d,\gamma}         |x|^{-\gamma}$          with
$C_{d,    \gamma}=\pi^{-\gamma+d/2}\Gamma(\gamma/2)/\Gamma((d-\gamma)/2)$
(see~(I.1.1.2) in~\cite{landkof72}). For every compact set $K$, define
\begin{equation}\label{def:capd=4}
    \Cap_{\gamma}(K):= \pars*{\inf_{\nu}\int k_\gamma(x-y) \, \nu(\rd  x) \nu(\rd y)}^{-1}, 
\end{equation}
where  the infimum  is  taken  over all  the  probability measures  with
support in $K$. 
In particular, the case $\gamma=d-2$ corresponds to the usual Newtonian
capacity.

For  $\gamma\in(0,d)$, the  function $\Cap_{\gamma}$  can be  extended
into  a  Choquet  capacity  (relative to  $\cK$)  on  $\cP(\r^d)$  using
again~\eqref{eq:choquet}        and~\eqref{eq:choquet-borel},        see
\cite[Section~II.1-2]{landkof72};  we  still  denote this  extension  by
$\Cap_{\gamma}$. We recall that $\Cap_{\gamma}$ is sub-additive, see
\cite[Eq.~(2.2.3)]{landkof72}: 
\[\Cap_{\gamma}(A\cup B) \leq
\Cap_{\gamma}(A)+ \Cap_{\gamma}(B),\quad A, B\subset \r^d. 
\]
   Let   us   stress
that~\eqref{eq:capacitable}  implies   that  Equation~\eqref{def:capd=4}
also holds  for all capacitable  sets, and  in particular for  all Borel
sets.

For    $\gamma\leq  d-2$,    we    mention    that,     according    to
\cite[Section~II.3]{landkof72},  the capacity  of a  compact set  $K$ is
equal to the capacity of its outer boundary, that is the boundary of the
unbounded connected component of $K^c$.

For  $\gamma\geq  d-2$, the  function $\Cap_{\gamma}$
is    strongly     sub-additive    by     \cite[Section~II.1,    Theorem
II.2.5]{landkof72}; so we could  also apply Theorem \ref{theo:ext-J} to
get its   extension   as a   Choquet   capacity. 

\end{remark}

\subsection{Brownian snake  capacities}\label{sec:BrwonianSnake}

Let
${\mathcal W}:=\{w :  [0, \zeta] \to \r^d \mbox{ is  continuous} \}$ be
the  space  of stopped  continuous  paths,  where $\zeta=\zeta_w\ge  0$
denotes the lifetime  of $w$. The end point, $w(\zeta)$,  is denoted by
$\widehat w$.  The space ${\mathcal W}$, equipped with the metric
\[
  d(w, w'):= \sup_{t\ge 0} |w(t\wedge \zeta_w)-  w'(t\wedge
   \zeta_{w'})|+ |\zeta_w- \zeta_{w'}|,
\]
is a Polish space. Let $x\in \r^d$ and ${\mathcal W}_x$ be the space of
stopped path  starting from  $w(0)=x$. The  Brownian snake,  denoted by
$(W_t)_{t\ge 0}$, is  a continuous strong Markov  process taking valued
in ${\mathcal W}_x$, where $W_0$ equals the trivial path (with $w(0)=x$
and  $\zeta_w=0$). The  excursion  measure outside  this trivial  path,
$\N_x$, can be characterized as follows, see \cite[Chapters 4 and 5]{LeGall1999}:
\begin{enumerate}[(i)]
\item 
The lifetime process $(\zeta_s)_{s\ge 0}$, under $\N_x$, is distributed
as a positive Brownian excursion under the It\^{o} measure. In
particular, with $T_0(\zeta):= \inf\{s>0: \zeta_s=0\}$ the length of the
excursion,
\[
  \N_x\Big(\sup_{0\le s\le T_0(\zeta)} \zeta_s > t\Big)= \frac{1}{2 t}, \qquad
  \N_x(T_0(\zeta) \in d t)= \frac{\rd t}{\sqrt{2\pi t^3}}
  \quad\text{for}\quad
  t>0. 
\]
 
\item  Conditionally  on $(\zeta_s)_{s\ge  0}$,  $(W_s)_{s\ge  0}$ is  a
  time-inhomogeneous continuous Markov process such that for any $s, s' > 0$, $W_s$
  and $W_{s'}$ coincide up to  $m_{s, s'}:=\inf_{t\in [s, s']} \zeta_t$,
  and   evolve    as   two    independent   Brownian    motions   during
  $[m_{s,   s'},  \zeta_{W_s}]$   and   $[m_{s,  s'},   \zeta_{W_{s'}}]$
  respectively.
 \end{enumerate}

 Under $\N_x$,  the range of the Brownian snake is defined by
 \begin{equation}
   \RR:=\{\widehat  W_s: 0\le s \le T_0(\zeta)\}= \{ W_s(t)\,\colon\,
   s\in [0,T_0(\zeta)],  \, t\in [0, \zeta_s]\} .
   \label{def-rangeBrownianSnake}
 \end{equation}

\medskip
Let $A$ be a Borel set of $\r^d$. 
 We define the function $u_A$ on $A^c$ by: 
\begin{equation}
\label{def-u_A}
u_A(x):= \N_x(\RR \cap A \neq \emptyset), \qquad x \in A^c.
\end{equation}
  Using the  scaling of  the Brownian  snake, one get  the existence  of a
regular conditional probability distribution  of the Brownian snake with
respect  to the excursion  length:
\[
  \N_x^{(1)}(\rd W)=\N_x(\rd W \, |\,
  T_0(\zeta)=1).
\]
By the scaling property of Brownian snake, we have 
\begin{align}  
\nonumber
u_A(x)
&= \int_0^\infty \N_0\Big(\RR \cap (A-x) \neq \emptyset \,|\,
         T_0(\zeta)=s\Big)\,  \frac{\rd s}{\sqrt{2\pi s^3}}\\ 
  &= \int_0^\infty \N_0^{(1)}\Big((s^{1/4}\RR)
    \cap (A-x) \neq \emptyset \Big)\,  \frac{\rd s}{\sqrt{2\pi s^3}}\cdot
\label{eq:scaling-ISE}   
\end{align}
We deduce the following scaling property
\begin{equation}\label{eq:scaling-Range}
  u_A(x)=a^2 u_{aA}(ax)
  \quad\text{for}\quad
    a>0.
\end{equation}
Moreover, according to  \cite[Theorem~1]{dlg1997},
Definition~\eqref{def-u_A} can be extended to $A\in \cP(\r^d)$
and, in dimension $d\geq 5$, 
there exists two constants $c_1$  and $c_2$ such that, 
if  $A\subset \ball(0,1)$ and
$|x|\geq 2$, then
\begin{equation}
   \label{eq:ineq-cap}
   c_1 \, \Cap_{d-4} (A) \leq  |x|^{d-2} \, u_A(x) \leq    c_2\,  \Cap_{d-4}   (A). 
\end{equation}
A set $A\subset  \r^d$ is called $\RR$-polar if $  \Cap_{d-4} (A)=0$, or
equivalently, if $u_A$ is identically zero.

We cite the following result from \cite{MR1341844}, and give a short proof for readers' convenience.
\begin{lemma}
  \label{lem:uA}
  Let $A$ be a Borel set. 
  The function $u_A$ is a nonnegative solution of
 \begin{equation}
   \label{eq:Du=u2}
   \Delta u= 4 \, u^2 \quad\text{in}\quad \overline A ^c.
 \end{equation}
When $A=K$ is compact, then $u_K$ is
the maximal nonnegative solution to~\eqref{eq:Du=u2}.
Furthermore, the function $u_A$  is  strictly  positive and  $C^\infty  $  on $\overline A^c$  unless $A$ is $\RR$-polar.
\end{lemma}

\begin{proof}
  According to Proposition~5.3  in \cite{MR1341844}, 
  when $K$ is compact, $u_K$ is the maximal nonnegative solution to~\eqref{eq:Du=u2}.

Let $A$ be a Borel set and $D$  be an open ball in $\overline A^c$. It suffices to prove \eqref{eq:Du=u2} in $D$. Let $X^D$
denote   the   exit  measure   of $D$ for  the Brownian   snake  (see   Section   2
in~\cite{LeGall1995} with  $\Omega=\r_+ \times D$). 
For $x\in D$, by \cite[Theorem~2.4]{LeGall1995} on the spatial Markov
property, we have
\[
  u_A(x)
=  \N_x\left(1- \mathrm{e}  ^{-\int_{\r^d} u_A(x)\, X^D(\rd x)}\right).
\]
Then we deduce by \cite[Corollary 4.3]{MR1341844}, that $u_A$
is $C^\infty$ in $D$ and 
solves~\eqref{eq:Du=u2} in $D$.

{Finally, we use~\eqref{eq:ineq-cap} to conclude  that the function $u_A$
is strictly positive on $\overline A^c$ unless $A$ is $\RR$-polar.}
\end{proof}

\subsection{Proof of Theorem \ref{p:snakecap}}

Recall that  $d\ge 5$,  and $\ball(x,r)$  denotes a  ball of  radius $r$
centered at  $x$.  By the scaling  property \eqref{eq:scaling-Range}, we
can  assume  that  $A\subset  \ball(0,1)$.  Let  $x\in  \r^d$,  $|x|\geq
4$. Following the ideas of~\cite{dlg1997}, we consider a smooth function
$\psi$ on $\r^d$  such that $\psi\in [0, 1]$, $\psi=0$  on $\ball(0, 2)$
and  $\psi=1$ outside  $\ball(0, 3)$.   Let $(\beta_t)_{t\geq  0}$ be  a
Brownian motion started  at $\beta_0=x$. Using It\^{o}  calculus, we get
that
\[
  (\psi u_A) (\beta_t)= u_A(x) + \int _0 ^t \nabla (\psi u_A) (\beta _s) \, \rd \beta _s +
  \frac{1}{2} \int _0 ^t \Delta(\psi u_A) (\beta _s) \, \rd s.
\]
Let $a> |x|$ and $T_a=\inf\{ y\geq 0\, \colon\, |\beta _t|=a\}$, which is
a.s.\ finite. We get that for $t\geq 0$,
\begin{equation}\label{eq:ito}
  \E\left[(\psi u_A) (\beta _{T_a \wedge t})\right]
  = u_A(x) + \frac{1}{2}  \E\left[\int_0^{T_a \wedge t} \Delta (\psi u_A)
    (\beta _s)\,\rd s\right]. 
\end{equation}
Since $T_a$ is finite and $\psi(\beta _{T_a})=1$, by \eqref{eq:ineq-cap}, we deduce that
\[
 \lim_{t\rightarrow \infty }
    \E\left[(\psi u_A) (\beta _{T_a \wedge t})\right]=
       \E\left[u_A (\beta _{T_a
        })\right]\le c_2a^{2-d}\Cap_{d-4}(A),
\]
thus
\[
 \lim_{a\rightarrow\infty}\lim_{t\rightarrow \infty }
    \E\left[(\psi u_A) (\beta _{T_a \wedge t})\right]
=0.
\]
By Lemma \ref{lem:uA}, we decompose $\Delta(\psi u_A)$ into the sum of
$\psi\Delta u_A=4\psi u_A^2$ and $f:=  u_A \Delta  \psi  + 2  \nabla  \psi \,  \nabla  u_A$.  By  monotone convergence, we have
\begin{align*}
   \lim_{a\rightarrow \infty } \lim_{t\rightarrow \infty }
  \E\left[\int_0^{T_a \wedge t}  (\psi \Delta u_A)
  (\beta _s)\,\rd s\right]
  &= 4 \lim_{a\rightarrow \infty } \lim_{t\rightarrow \infty }
  \E\left[\int_0^{T_a \wedge t}  (\psi  u_A^2)
    (\beta _s)\,\rd s\right]\\
    &= 4 
  \E\left[\int_0^{\infty }  (\psi  u_A^2)
      (\beta _s)\,\rd s\right]\\
  &= 4 c_d \int_{\r^d}  (\psi  u_A^2)(y) |x-y|^{2-d}\, \rd y,
\end{align*}
where  
\[
c_d:= \frac{ \, \Gamma(\frac{d}2 -1)}{2 \pi^{d/2}}\cdot
\]
The latter integral is finite  since $\psi  u_A^2$ is bounded on
$\r^d$, and thus on $\ball(x, 1)$, and since 
\begin{equation}
   \label{eq:bd-psiu4}
  (\psi  u_A^2)(y) \leq  c\,
  (1 \vee |y|)^{4 -2d}
  \quad\text{for}\quad
  y\in \r^d
  \end{equation}
by~\eqref{eq:ineq-cap}, which entails that for some constant $c'>0$,
\begin{equation}
   \label{eq:intu4-fini}
  \int_{\ball(x,1)^c}   (\psi  u_A^2)(y) |x-y|^{2-d}\, \rd y \leq
  c' \int_0^\infty  (1 \vee r) ^{4 - 2d} \, r^{d-1}\, \rd r <\infty .
\end{equation}

Then, notice the function $f= u_A \Delta  \psi  + 2  \nabla  \psi \,  \nabla  u_A$ is continuous with support in $\overline \ball(0, 3)
\cap \ball(0,2)^c$. Since the time spent by the Brownian motion in $\ball(0, 3)$ has
a finite expectation, we deduce by dominated convergence that
\begin{equation}\label{eq:f_finite}
   \lim_{a\rightarrow \infty } \lim_{t\rightarrow \infty }
 \E\left[\int_0^{T_a \wedge t}  f
   (\beta _s)\,\rd s\right]
 =  
 \E\left[\int_0^\infty   f
   (\beta _s)\,\rd s\right]= c_d \int_{\r^d} f(y)  |x-y|^{2-d}\, \rd
 y\in \r.
\end{equation}

In conclusion we obtain from \eqref{eq:ito} that
\begin{equation}\label{eq:uK=}
  2 u_A(x)= - c_d \int_{\r^d} (4 \psi  u_A^2+ f)(y)\,  |x-y|^{2-d}\, \rd
 y.
\end{equation}

To conclude the proof of Theorem \ref{p:snakecap}, it is enough to prove
that 
\begin{equation}
  \label{eq:lim-uA}
 \lim_{x \rightarrow \infty }
|x|^{d-2} u_A(x) = -  \frac{c_d}{2}
  \int_{\r^d} (4 \psi u_A^2  - u_A \Delta \psi)(y)\, \rd y<\infty.
\end{equation}

To begin with, since the function $f$ is bounded with bounded support, we deduce by dominated convergence that
\begin{equation}\label{eq:f}
  \lim_{x\rightarrow \infty } |x|^{d-2} \int_{\r^d} f(y)\,
  |x-y|^{2-d}\, \rd y 
=  \int_{\r^d} f(y)\,\rd y<\infty. 
\end{equation}

Let $\varepsilon\in (0,\frac 1 2)$. Thanks to~\eqref{eq:bd-psiu4}, we know that
$     \psi     u_A^2$     is     integrable     on     $\r^d$.     Since
$y\mapsto \ind_{\ball(0, \varepsilon |x|)}(y)\,  |x|^{d-2} |x-y|^{2-d}$ is bounded
uniformly in  $x$ and converges as  $x\rightarrow \infty $ to  $1$, we
deduce by dominated convergence that
\begin{equation}
   \label{eq:cvuk-1}
\lim_{x\rightarrow \infty }       
 |x|^{d-2}  \int_{\ball(0, \varepsilon |x|)} \psi u_A^2 (y)
    |x-y|^{2-d} \, \rd y = \int_{\r^d} \psi u_A^2 (y)
    \, \rd y<\infty.
\end{equation}

For  $y\not \in \ball(0, \varepsilon |x|)$, by \eqref{eq:bd-psiu4}, we have
\[  \psi     u_A^2(y)
|x|^{d-2}|x-y|^{2-d} \lesssim
  |y|^{4 -2d} |x|^{d-2} |x-y|^{2-d}=: h(x,y).
\]
For  $y\in \ball(x, \varepsilon |x|)$, we have $h(x,y)\leq  (1-\varepsilon)^{4
  -2d} |x|^{ 2-d} |x-y|^{2-d}$, hence
\[
  \int_{\ball(x, \varepsilon |x|)} h(x,y)\, \rd y \lesssim \varepsilon ^{2} |x|^{4 -d}.
\]
For $y\not \in \ball(x, \varepsilon |x|) \cup \ball(0, \varepsilon |x|)$, we
also get that  $h(x,y)\leq  \varepsilon^{2-d}  |y|^{ 4-2d}$ and thus 
\[
  \int_{\ball(x, \varepsilon |x|)^c \cap \ball(0,\varepsilon |x|)^c } h(x,y)\,
  \rd y \lesssim \varepsilon^{6-2d} \,  |x|^{4 -d}.
\]
From the above we conclude that $\lim_{|x|\rightarrow \infty } \int_{\ball(0, \varepsilon
  |x|)^c} h(x,y)\, \rd y =0$.
This together with \eqref{eq:uK=}, \eqref{eq:f} and~\eqref{eq:cvuk-1} gives that  
\[
 \lim_{x \rightarrow \infty }
 |x|^{d-2} u_A(x)
 = - \frac{c_d}{2}  \int_{\r^d} (4 \psi  u_A^2+ f)(y)\,  \rd
 y<\infty.
\]
Performing  an integration  by part  and using  that the  derivatives of
$\psi$  are  zero  outside  $\overline  \ball(0,  3)$,  we  deduce  that
$\int _{\r^d} (\nabla \psi \nabla u_A) (y) \, \rd y = - \int_{\r^d} (u_A
\Delta \psi)(y)  \, \rd y$.  This  gives~\eqref{eq:lim-uA} and concludes 
the proof.  \hfill $\Box$

\begin{remark}\label{rmk:low_dim}
We now consider low dimension case, that is,  $d\in\{1,2,3,4\}$. We obtain by \cite[Eq.~(4.7)]{iscoe} that for all $r>0$
\begin{equation}\label{eq:d_le_4}
\lim_{x \rightarrow \infty} \varphi_d(|x|)\, u_{\ball(0,r)}(x)=1,
\end{equation}
where the function $\varphi_d$ is defined  for $t>1$ by
\[
\varphi_d(t)= \begin{cases}
\frac{2}{4-d}\, t ^{2},\,&1\le d\le 3,\\
 2\, t^{2}\log(t),\,&d=4.
\end{cases}
\]
Notice the function $\varphi_d$ does not depend on $r$. Let  $A\subset\r^d$ be a bounded Borel set with nonempty interior. There exists  $x_0\in A$ and $r_1,r_2>0$ such that $\ball(x_0,r_1)\subset A\subset\ball(x_0,r_2)$. We thus deduce from~\eqref{eq:d_le_4} that
 $\N_x(\RR\cap A\neq \emptyset)=u_A(x)$ satisfies the same asymptotic as $u_{\ball(0,r)}(x)$:
 \[
 \lim_{x \rightarrow \infty} \varphi_d(|x|)\, u_{A}(x)=1.
\]
\end{remark}

\subsection{Proof of Proposition~\ref{prop:propertiesBScap}}
 Recall $d\geq 5$ and  $\cK$ denotes the collection of compact sets in $\r^d$.
We first prove Point~\ref{it:BS=cap} of Proposition~\ref{prop:propertiesBScap} on $\BScap$ being a strongly
sub-additive Choquet capacity by considering the extension $J^*$ of
$J:=\BScap|_\cK$ as in Section
  \ref{sec:capacity}. 
Notice that for every $K\in\cK$, $J(K)$ is finite.

\begin{lemma}\label{lem:BScap=cap}
  Let $d\geq  5$.  Let $J^*$ denote  the extension of $J$ defined by 
using~\eqref{eq:choquet} and~\eqref{eq:choquet-borel}.  Then $J^*$  is  a  strongly sub-additive  Choquet
  capacity which coincides with $\BScap$ on bounded Borel sets.
\end{lemma} 
\begin{proof}
For $K,K'\in\cK$ and  $x$  large enough, we trivially deduce from~\eqref{def-u_A} and inclusion-exclusion principle that 
$0\leq u_K(x)\leq u_{K\cup K'}(x)$  and
$u_{K\cup K'}(x) + u_{K\cap K'}(x) \leq u_K(x)+  u_{K'}(x)$, hence $J$ is non-decreasing and strongly sub-additive.  

Set  $\open_n=  \{x\in  \r^d\,  \colon\,  d(x,  K)<1/n\}$  which is  open  and
$K_n=\overline   \open_n$   which is   compact.    Since  the   map   $J$   is
non-decreasing, we  deduce that  $J(K)\leq J(K_n)$.  Moreover,  $\RR$ is
$\N_x$-a.e.\  compact, so
\[
  \ind_{\{\RR\cap K\ne\emptyset\}}
  =\lim_{n\rightarrow\infty}
  \ind _{\{\RR\cap K_n\ne\emptyset\}},\quad \text{$\N_x$-a.e..}
\]
Hence by dominated convergence in \eqref{def-u_A} and \eqref{eq:lim-uA},
we have  for $x$ large enough
\[
  \lim_{n\rightarrow\infty  }   u_{K_n}(x)=u_K(x)
  \quad\text{and then}\quad
   \lim_{n\rightarrow\infty   }    J(K_n)=J(K).
\]
Thus for every $a>J(K)$, we can find $n$ such that $a>J(K_n)$, and any $K'\subset \open_n$ must satisfy $J(K')\le J(K_n)<a$, showing the right continuity  of $J$ on $\cK$. 
We conclude by Theorem~\ref{theo:ext-J}
that $J^*$ is a Choquet capacity which coincides with $J$
on $\cK$. 

We also deduce from Remark~\ref{rem:sub-+J} that $J^*$ is strongly
sub-additive. 

\medskip

Finally, we show that $J^*$ and $\BScap$ agrees on every bounded Borel set $A$. Indeed, since
\[
\sup_{K\in\cK,K\subset A}\ind_{\{\RR\cap K\ne\emptyset\}}\le
\ind_{\{\RR\cap A\ne\emptyset\}}
=\sup_{x\in A}\ind_{\{\RR\cap \{x\}\ne\emptyset\}}
\le
\sup_{K\in\cK,K\subset A}\ind_{\{\RR\cap K\ne\emptyset\}},
\]
we deduce by dominated convergence as above that for $x$ large enough
\begin{equation}\label{eq:compact_sup_1}
 \sup_{K\in\cK,K\subset A}   u_{K}(x)=u_A(x)\quad\text{and then} \quad
 \sup_{K\in\cK,K\subset A}    \BScap(K)=\BScap(A).
\end{equation}
Since all Borel sets are capacitable (with respect to $J^*$), we also have
\[
\sup_{K\in\cK,K\subset A}    J^*(K)=J^*(A).
\]
Hence $J^*|_\cK=J=\BScap|_\cK$ implies that $J^*$ agrees with $\BScap$
on bounded Borel sets.
\end{proof}

We  deduce  Point~\ref{it:BS=scaling} of Proposition~\ref{prop:propertiesBScap} on the  scaling  and  translation
invariance properties for  bounded Borel sets from  the scaling property
of    the    Brownian     snake,    see~\eqref{eq:scaling-Range},    and
Point~\ref{it:BS=d-4}  on   the  comparison  with  the   Riesz  capacity
$\Cap_{d-4}$   from~\eqref{eq:ineq-cap},  then   we  can   extend  those
properties  first  to  bounded sets  using~\eqref{eq:choquet-borel}  and
general sets using the monotonicity of Choquet capacities.

Point~\ref{it:BS=reg}    corresponds   to    the    next   lemma.    For
$\open\subset \r^d$ open, let  $\partial_{\text{ext}}\open$ denote its
outer boundary,  that is  the boundary of  the unbounded  open connected
component of $\overline{\open}^c$.

  \begin{lemma}\label{lem:regular}
Let $d\geq 5$ and $D\subset\r^d$ be  bounded open set such that,  for all $y\in
\partial_{\text{ext}}\open$,  the quantity 
$ \liminf_{n\rightarrow \infty } 2^{n(d-2)} \, \Cap_{d-2} \left(\open  \cap
   \ball(y, 2^{-n}) \right) 
$ is positive.
Then, we have
\[
\BScap(\open)=\BScap(\overline{\open}).
\]
\end{lemma}

\begin{proof}
In a first step, using the notion of super-regularity from
\cite{dlg1997}, we shall prove that
\begin{equation}
  \label{eq:super-reg}
   \N_y(\RR\cap D\neq \emptyset)=\infty
\quad\text{for all}\quad
y\in \partial_{\text{ext}}\open. 
\end{equation}
Let $y\in \partial_{\text{ext}}\open$ be fixed. For simplicity, write
$\ball_n=\ball(y, 2^{-n})$ for $n\in \n$ and set
\begin{equation}
   \label{eq:reg-x-D}
c_0= 4^{-1} \liminf_{n\rightarrow \infty } 2^{n(d-2)} \, \Cap_{d-2}
\left(\open  \cap 
      \ball_n\right),  
  \end{equation}
  which   is   positive   by   assumption.  By   scaling, we have
  $\Cap_{d-2} (B_n)= 2^{-n(d-2)} \,  \Cap_{d-2} (B_0)$.  Let $k\in \n^*$
  be  such  that  $2^{-k(d-2)}  \Cap_{d-2}  (B_0)  \leq  c_0$  and
  thus, by scaling, 
  $\Cap_{d-2} (B_{n+k})\leq c_0 2^{-n(d-2)}$ for all $n\in \n$.  We consider the spherical
  shell $C_{n,k}=\ball_n \cap \ball_{n+k}^c$.  By sub-additivity and scaling
  of $\Cap_{d-2}$, we have for $n$ large enough
\begin{align*}
\Cap_{d-2}\left(\open\cap C_{n,k}\right)
&  \geq    \Cap_{d-2}\left(\open\cap \ball_n\right)
  - \Cap_{d-2}\left(\open\cap \ball_{n+k}\right)
  \\
  & \geq   2 c_0 \, 2^{-n(d-2)}
    - \Cap_{d-2}\left( \ball_{n+k}\right) \\
  & \geq  c_0\, 2^{-n(d-2)}.   
\end{align*}
 Since  clearly 
$  \Cap_{d-4}(A) \geq  \Cap_{d-2}(A)$ for $A\subset \ball(0, 1)$,  we deduce
 by sub-additivity that:
\begin{align*}
  \sum_{n=0}^\infty 2^{n(d-2)}\,  \Cap_{d-4} (D \cap C_{n,1})
  & \geq   \sum_{n=0}^\infty 2^{n(d-2)}\,  \Cap_{d-2} (D \cap C_{n,1})
  \\
  & \geq   \sum_{\ell=0}^\infty 2^{\ell k (d-2)}\, \sum_{j=0}^{k-1} 
    \Cap_{d-2} (D \cap C_{\ell k+j,1}) 
  \\
    & \geq   \sum_{\ell=0}^\infty 2^{\ell k (d-2)}\, 
      \Cap_{d-2} (D \cap C_{\ell k, k})
  \\
  &=\infty .
\end{align*}
Thanks to~\cite[Theorem~3]{dlg1997}, we deduce that $y$ is super-regular
for $\open$, and thanks to Proposition~2 therein with the fact that $y\not
\in \open$, we get that~\eqref{eq:super-reg} holds. 

\medskip

Let $D'$ be the unbounded open connected component of
$\overline{\open}^c$ and $x\in D'$ be fixed. 
In a second step we shall prove that
\begin{equation}
  \label{eq:range-DD'}
   \N_x(\RR\cap \overline {\open}\neq \emptyset)= \N_x ( \RR\cap \open\neq \emptyset).
\end{equation}
Using for the inclusion the spatial Markov property of the Brownian snake, see
\cite[Theorem~2.4]{LeGall1995}, we get that $\N_x$-a.e.\ 
\[
  \{\RR\cap \overline {\open}\neq \emptyset\}= \{\RR\cap \open'^c\neq \emptyset\}  \subset \{X^{\open'} \neq 0\}, 
\]
where  $X^{\open'}$ is  the exit  measure of  $\open'$ for  the Brownian
snake      (see      Section       2      in~\cite{LeGall1995}      with
$\Omega=\r_+  \times \open'$).   As~\eqref{eq:reg-x-D} holds,  we deduce
from   ~\cite[Theorem~6.9]{LeGall1999}     (by    first    considering
$\open'\cap \ball(x,r)$ and then letting  $r$ goes to infinity and using
that $\RR$ is compact) that
\begin{equation}
   \label{eq:XD=RD-1}
  \N_x(\RR\cap \overline {\open}\neq \emptyset)=\N_x(X^{\open'} \neq 0).
\end{equation}

According to
  \cite[Lemma~2.1]{dynkin91}\footnote{
 Notice that \cite[Lemma~2.1]{dynkin91} is stated under $\p_\mu$ with
$\mu$ a finite measure on $\r^d$; and $\p_\mu$ can be seen as the distribution of
the Poisson point measure $\sum_{i\in \I} \delta_{W^i}(\rd W)$ of Brownian
snakes with intensity $\mu(dx)\, \N_x(\rd W)$. Using again the spatial
Markov property of the Brownian snake under $\N_x$ with the exit measure
 $X^B$ of $B$,  an open ball centered at $x$  whose closure lies
in $\open$, and Lemma~2.1 with $\mu=X^B$, we deduce
that~\eqref{eq:XD-RR-2} holds indeed under $\N_x$.}, we also have that $\N_x$-a.e.\ 
\begin{equation}
   \label{eq:XD-RR-2}
  \{X^{\open'}\neq 0\} \subset \{\RR\cap \open \neq \emptyset\}.
\end{equation}
Since the support of $X^{\open'}$ is a subset of $\partial \open'=
\partial_{\text{ext}}\open$, we 
deduce from~\eqref{eq:super-reg} that  $\int
\N_y(\RR
\cap \open\neq \emptyset) \, X^{\open '}(\rd y)=\infty $ if and only if
$X^{\open '}\neq
0$. By the spatial Markov property, we get that
\begin{equation}
   \label{eq:XD=RD-2}
  \N_x(\RR  \cap  \open \neq  \emptyset)
  =  \N_x\left(\int  \N_y(\RR  \cap \open \neq
    \emptyset) \, X^{\open '}(\rd y) >0 \right)
  =\N_x (X^{\open '} \neq 0).
\end{equation}
We thus deduce~\eqref{eq:range-DD'} from~\eqref{eq:XD=RD-1}
and~\eqref{eq:XD=RD-2}.
Then use the definition of $\BScap$ to conclude. 
\end{proof}

To conclude, Point~\ref{it:BS=6} is a direct consequence of the next
remark. 

\begin{remark}\label{rem:u1}
By \eqref{eq:scaling-Range} and symmetry, there is a function $u$
defined on $(1, \infty )$ such that
$u_{\ball(0,r)}(x)=r^{-2}u(r^{-1}|x|).$
By Lemma \ref{lem:uA}, $u$ is the maximal solution
on $(1, \infty )$ of
\begin{equation}
\label{eq:nonlin}
    u''(t) +
    \frac{(d-1)}{t}\,  u'(t) = 4 u(t)^2.
\end{equation} 
As in \cite[Section~5]{Delmas99}, we can expand it as 
  $u(t)=t^{2-d}  \sum_{n\geq 0} a_n  t^{-n(d-4)}$, where
  \[
    a_n=\frac{4}{n \delta (n\delta +1)} (d-2)^{-2} \sum_{k=0}^{n-1}a_k
    a_{n-k-1} \quad\text{for}\quad
    n\ge 1,
    \quad\text{and}\quad
     \delta=\frac{d-4}{d-2}\cdot
  \]
Solving \eqref{eq:nonlin} then reduces to determine $a_0>0$, the maximum number so that the series above converges when $t>1$. 
In particular, we have
$a_0=\lim_{t\rightarrow \infty } t^{d-2}u(t)=\BScap(\ball(0, 1))$, which is not known in general. 

Let us mention that for $d=6$,
\begin{equation}
  \label{eq:u-for-d6}
  u(t)=\frac{6}{(t^2-1)^{2}} \quad\text{for}\quad
  t>1
\end{equation}
is explicit  (this result seems to  be new in the  literature), and thus
$ \BScap(\ball(0, 1))=6 $.
\end{remark}

\section{Convergence of the branching capacity:  Proof of Theorem \ref{prop:A^epsilon}} \label{sec:convergence4}

To begin with, we recall a  comparison result between the discrete range
$\brwrange_c$ of the  branching random walk indexed by a  critical Galton-Watson tree
$ \ttree c$ and the Brownian  snake range $\RR$ under the normalized
excursion measure $\N^{(1)}$.
We define the Hausdorff distance between two  nonempty  compact   sets
$K, K'\subset\r^d$ by
\[
\hdist(K,K')=\max\pars*{\sup_{x\in K}d(x,K')\, ,\,     \sup_{x\in
    K'}d(x, K')}.
\]

By  Janson and  Marckert \cite[Theorem  2]{JM05}  (see Le  Gall and  Lin
\cite[formula  (18)]{legall_lowdim} for   the  statement  in  present
setting), on a common probability  space $(\Omega, {\mathcal F}, \P)$ we
may construct  a random set  $\RR_* \subset \r^d$, distributed  as $\RR$
under $\N_0^{(1)}$,  and a  sequence $(\brwrange_{(n)})_{n\ge 1}$  of random
sets in $\z^d$  such that for any $n$,  $\brwrange_{(n)}$ is distributed
as $ \brwrange_c $ conditioned on $\{\#\ttree c=n\}$ under $\P_0$, and almost surely,
\begin{equation} \label{thm:jm05}   
  \hdist\pars*{n^{-1/4} \brwrange_{(n)} ,\, \coeffm\RR_*}\rightarrow0
  \quad\text{as}\quad n\to\infty, 
\end{equation}
where, with $\sigma$ and $\brwm$ defined in~\eqref{hyp-tree}  and \eqref{hyp-brw}, 
\[
\coeffm=\left(\frac{2}{\sigma}\right)^{1/2} \, \brwm^{1/2}.
\]

Now we are ready to prove Theorem \ref{prop:A^epsilon} on admitting
Theorem \ref{theo:compareBcap}. Heuristically, we first rephrase
branching capacity in terms of hitting probability. Then we show that
when the (discrete) branching random walk has diameter of order $n$,
typically there are about $n^4$ vertices on the tree. Finally, we
compare discrete and continuous hitting events (up to the error of an
$\varepsilon$-neighborhood) by  \eqref{thm:jm05}. 

We recall that $o_n(1)$  denotes a function of $n$ which  goes to $0$ as
$n$ goes to infinity.

\begin{proof}[Proof of Theorem \ref{prop:A^epsilon} by admitting Theorem
  \ref{theo:compareBcap}.] 
  The  proofs  of  \eqref{upper:nK}  and  \eqref{lower:Open}  need  some
  preparations. At first we may assume for simplicity that the offspring
  distribution $\mu$ is aperiodic, as one  can easily adapt the proof
  line  by   line  for   the  periodic   case.   In   particular,  given
  aperiodicity,  by  Dwass  \cite{dwass}  and the  local  central  limit
  theorem (see \cite[Theorem
  2.3.9]{lawler2010random}), we have
\begin{align}\label{eq:ctree} 
  \P(\#\ttree c=n) = \Big(\frac{1}{\sigma\sqrt{2\pi}}+o_n(1)\Big)\,  n^{-3/2}.
\end{align}

Secondly,   let   $K\subset\r^d$   be   a   given   compact   set.   Let
$x\in \r^d\backslash \{0\}$, with $|x|$  large enough, and $\eta\in(0,1/2)$ be such
that, with $\alpha$ and $C$ from Theorem \ref{theo:compareBcap},
\begin{equation}
  \label{K-x-eta}
  K\subset\ball(0, \eta' |x|) \quad\text{with}\quad
  \eta=C\eta'^\alpha.
\end{equation} 
By Theorem \ref{theo:compareBcap}, for every $n\ge 1$   we have
\begin{equation}  
\label{eq:doubleeq}
\frac{\P_{\lfloor nx\rfloor}
  \pars*{\brwrange_c \cap n K\neq \emptyset}}{(1+\eta) g(\lfloor
  nx\rfloor)}
\le \Bcap(nK) \le \frac{\P_{\lfloor nx\rfloor}
  \pars*{\brwrange_c \cap nK\neq \emptyset}}{(1-\eta) g(\lfloor
  nx\rfloor)}\cdot
\end{equation}

{\bf Proof of \eqref{upper:nK}.} It suffices to show that for any $K$ compact set and $x\neq0$ satisfying $K\subset \ball(0, |x|/2)$   and  $\varepsilon>0$, \begin{equation}  \label{inter-nK-up1} \limsup_{n\rightarrow\infty}n^2\, \P_{\lfloor nx\rfloor}\pars*{\brwrange_c\cap nK\neq \emptyset}
\le \frac{1}{\sigma}\,  \N_{\coeffm^{-1}x}\Big(\RR \cap   (\coeffm^{-1}K)^{\varepsilon} \neq \emptyset\Big).
 \end{equation}
Indeed,  for any $\eta\in(0, \frac12)$, let $|x|$ be large enough such
that \eqref{K-x-eta} holds.  Applying the upper bound in
\eqref{eq:doubleeq} and using  the asymptotic of $g$ in
\eqref{c_g}, Equation~\eqref{inter-nK-up1} yields that 
\begin{align*}
  \limsup_{n\rightarrow\infty}\frac{\Bcap(nK)}{n^{d-4}}
  & \le \frac {1+\eta}
    {(1-\eta)\sigma c_{g}}\,
    \frac{\absolute*{x}_\theta^{d-2}}{\absolute*{\coeffm^{-1} x}^{d-2}}\,  
    \absolute*{\coeffm^{-1} x}^{d-2}
    \N_{\coeffm^{-1}x}\Big(\RR \cap   (\coeffm^{-1}K)^{\varepsilon}
    \neq \emptyset\Big).
\end{align*}
Recall that $\coeffm=\left(\frac{2}{\sigma}\right)^{1/2} \, \brwm^{1/2}$. Letting $x\rightarrow\infty$ and
then $\eta\rightarrow 0$, we deduce from Theorem \ref{p:snakecap} and
the scaling properties of $\BScap$, see \eqref{eq:BScap_scaling}, that 
\begin{equation}  \label{eq:BCapdouble}
\limsup_{n\rightarrow\infty}\frac{\Bcap(nK)}{n^{d-4}} \le  \frac {2}
{\sigma^2{c_{g}}}\,  \BScap(\brwm^{-1/2}\cnbd K\varepsilon).
\end{equation}
Then    let     $\varepsilon\to    0$     and    use     the    monotone
property~(\ref{it:cap-Kdec})    of   the    Choquet   capacity    to   get
\eqref{upper:nK}.

\medskip

It remains to show \eqref{inter-nK-up1}. 
For any $\delta \in (0,1)$, we have
\begin{align}
\P_{\lfloor nx\rfloor}\pars*{\brwrange_c \cap nK\neq \emptyset}
=&\Big(\sum_{j<  \delta n^4} + \sum_{\delta n^4 \le j \le \delta^{-1} n^4} + \sum_{j> \delta^{-1} n^4}\Big)\, \P_{\lfloor nx\rfloor}\pars*{\brwrange_c \cap nK \neq \emptyset , \,   \#\ttree c=j}
\nonumber\\
=: &\eqref{Tc3terms0}_1+ \eqref{Tc3terms0}_2+ \eqref{Tc3terms0}_3. \label{Tc3terms0}
\end{align}

We claim that when $\delta$ is small, both $\eqref{Tc3terms0}_1 $ and  $ \eqref{Tc3terms0}_3$ are negligible in the sense that they do not contribute to the $\limsup$ term in \eqref{inter-nK-up1}. 
Indeed,  by \eqref{eq:ctree},  we have 
$$  \eqref{Tc3terms0}_3
\le
\sum_{j> \delta^{-1} n^4} \P\pars*{\#\ttree c=j} 
\lesssim \sum_{j> \delta^{-1} n^4} j^{-3/2}  \lesssim \delta^{1/2} n^{-2}, $$
yielding that \begin{equation}  \limsup_{\delta\to 0+} \limsup_{n\to\infty} \big(n^2 \times \eqref{Tc3terms0}_3 \big)=0.  \end{equation}

 For $\eqref{Tc3terms0}_1$, we remark that
 \[
   \P_{\lfloor nx\rfloor}\pars*{\brwrange_c \cap nK \neq \emptyset , \,   \#\ttree c=j}
=\P_{0}\pars*{\brwrange_c \cap (nK- \lfloor nx\rfloor) \neq \emptyset , \,   \#\ttree c=j}.
\]
Since $K\subset \ball(0, |x|/2)$,  for any $y \in nK- \lfloor
nx\rfloor$, we have $|y|\ge n |x|/2 - \sqrt{d} >
n|x|/3$ for all  $n\ge n_0(x)$ and $n_0(x)$ large enough. Therefore, we have 
\begin{eqnarray} 
\eqref{Tc3terms0}_1 
&\le&  \sum_{j<  \delta n^4}  \P_{0}\pars*{\max_{z\in \brwrange_c} |z| \ge  \frac{n|x|}3, \, \#\ttree c =j}
 \nonumber \\
 &\lesssim &
\sum_{j=1}^{\delta n^4} j^{-3/2} \P_{0}\parsof*{\max_{z\in \brwrange_c} |z| \ge  \frac{n|x|}3}{ \#\ttree c=j}, \nonumber
\end{eqnarray}
where the last line follows from \eqref{eq:ctree}.
Using  \cite[Eq.~(4.24)]{bai2022convergence} and Markov's inequality, we see that for all $j$
\[
\P_{0}\parsof*{\max_{z\in \brwrange_c} |z| \ge \frac{n|x|}3}{ \#\ttree c=j} \lesssim (n|x|)^{-5} j^{5/4}.
\] 
We deduce that 
\[
\eqref{Tc3terms0}_1  \lesssim  \sum_{j=1}^{\delta n^4} j^{-3/2} \frac{j^{5/4}}{(n|x|)^5}
\lesssim \delta^{3/4} n^{-2} |x|^{-5}.
\]
This implies that
\begin{equation}
  \limsup_{\delta\to 0+} \limsup_{n\to\infty} \big(n^2 \times
  \eqref{Tc3terms0}_1 \big)=0.
\end{equation}

For $\eqref{Tc3terms0}_2$, 
by \eqref{eq:ctree} again,
we have 
\begin{equation}   \label{eq:Tc3terms0-up}
\eqref{Tc3terms0}_2= \pars*{\frac{1}{\sigma\sqrt{2\pi}}+o_n(1)} \,
\sum_{\delta n^4 \le j \le \delta^{-1} n^4} j^{-3/2} \P
\pars*{\brwrange_{(j)}\cap (nK-\lfloor 
  nx\rfloor) \neq \emptyset }.   
 \end{equation}

 By \eqref{thm:jm05}, we have
\[
  \max_{\delta n^4 \le j \le \delta^{-1} n^4} \P\pars*{d_H(j^{-1/4}
    \coeffm^{-1}\brwrange_{(j)},  \RR_*) > \varepsilon \delta^{1/4}/2}
  =o_n(1),
\]
where $o_n(1)$ may depend  on $\varepsilon, \delta$ but $o_n(1) \to 0$
as $n\to \infty$, and $\P$ is the coupling law that ensures \eqref{thm:jm05}. 
   To shorten notations, for $K, K'\subset \r^d$ compact, we write
   \[
     \dmin(K, K')= \min \{|x-x'|\,\colon\, x\in K, x'\in K'\}.
   \]
   Notice that  $\dmin$ is not a distance. We will use that
   for $K, K', K''\in \cK$,
  $x\in \r  ^d$ and $\varepsilon\geq 0$
   \begin{align*}
     &\dmin(K, K')\leq  \varepsilon \, \Longleftrightarrow\, 
       K^\varepsilon\cap K'\neq \emptyset \, \Longleftrightarrow\,  K \cap
       K'^\varepsilon\neq \emptyset,\\
     &    \dmin (K, K'+x)\leq  \dmin (K', K')+ |x|,\\
    &    \dmin (K, K'')\leq  \dmin (K', K'')+ d_H(K, K').
   \end{align*}
   Note that for all large $n$ and $\delta n^4 \le j \le \delta^{-1}
   n^4$,
\begin{multline*}
   \P\pars*{\brwrange_{(j)}\cap (nK -\lfloor nx\rfloor) \neq \emptyset }
\\ 
\begin{aligned}
   &=
\P\pars*{\dmin(  \coeffm^{-1} \brwrange_{(j)}, \coeffm^{-1}(nK-
  \lfloor nx\rfloor)) =0} 
\\
&\le
\P\pars*{\dmin (n^{-1}j^{1/4} \RR_*,    \coeffm^{-1}( K-x )) \le
  \varepsilon } +  \P\pars*{d_H(\coeffm^{-1}\brwrange_{(j)}, j^{1/4}
  \RR_*) > \varepsilon (\delta j)^{1/4}/2}\\ 
&=
\P\pars*{\dmin (n^{-1}j^{1/4} \RR_*,    \coeffm^{-1}( K-x )) \le \varepsilon } +  o_n(1),
\end{aligned}
\end{multline*}
where   in   the    first   inequality   we   have    used   that,   for
$R_j:=\coeffm^{-1}    \brwrange_{(j)}$, $ K':=
\coeffm^{-1}( nK-\lfloor nx\rfloor )$ and $   R_*:=j^{1/4}    \RR_*  $, we have
\[
  \dmin(R_j,K')=0 \implies \dmin (R_*, K') \leq  d_H(R_j, R_*),
\]
and thus on $\{\dmin(  \coeffm^{-1}
\brwrange_{(j)}, \coeffm^{-1}(nK- \lfloor nx\rfloor)) =0\}$, we have
for $\varepsilon n \geq 2 \sqrt{d}$
\begin{align*}
    d_H(\coeffm^{-1}\brwrange_{(j)}, j^{1/4}  \RR_*) \le \varepsilon
  (\delta j)^{1/4}/2
&  \implies
    \dmin (j^{1/4} \RR_*,    \coeffm^{-1}(
     nK-\lfloor nx\rfloor)) \le \varepsilon (\delta j)^{1/4}
   /2\\
  &  \implies
    \dmin (n^{-1}j^{1/4} \RR_*,    \coeffm^{-1}(
    K-x ))
    \le \varepsilon.
\end{align*}
In the same way,  we have
\begin{align}
\P\pars*{\brwrange_{(j)}\cap (nK^{2\varepsilon}-\lfloor nx\rfloor) \neq \emptyset }
 &= \P\pars*{\dmin(   \coeffm^{-1}\brwrange_{(j)},  \coeffm^{-1}( nK-
   \lfloor nx\rfloor)) \le 2\varepsilon }  
 \nonumber \\
& \ge
\P\pars*{\dmin(n^{-1}j^{1/4} \RR_*,     \coeffm^{-1}(K-x) ) \le
                \varepsilon }
   \nonumber \\
& \hspace{2cm}- \P\pars*{d_H(\coeffm^{-1}\brwrange_{(j)}, j^{1/4}  \RR_*) > \varepsilon (\delta j)^{1/4}/2}
\nonumber \\
&\ge 
\P\pars*{\dmin(n^{-1}j^{1/4} \RR_*,     \coeffm^{-1}(K-x) ) \le \varepsilon } + o_n(1). 
\label{lower:inter-2}
\end{align}
We will use \eqref{lower:inter-2} in the proof of  \eqref{lower:Open}. 
 Going back to \eqref{eq:Tc3terms0-up}, we obtain that 
\begin{align*} 
  \eqref{Tc3terms0}_2
  &\le
\frac{1}{\sigma\sqrt{2\pi}} \, \sum_{\delta n^4 \le j \le \delta^{-1}
    n^4} j^{-3/2}\, \P\pars*{\dmin(n^{-1}j^{1/4} \RR_*,
    \coeffm^{-1}( K-x )) \le \varepsilon }  + o(n^{-2}) \\
  &\le \frac{1}{\sigma\sqrt{2\pi}} \, \int_0^\infty t^{-3/2}
    \P\pars*{\dmin(n^{-1} t^{1/4} \RR_*,
    \coeffm^{-1}( K-x )) \le \varepsilon  } \, \rd t 
+ o(n^{-2}) ,
\end{align*}
where in the last inequality we have replaced the sum over $j$ by the integral over $t$ by monotonicity. 
Using a change of variables $t= n^4 s$ and that $\RR_*$ under $\P$
is distributed as $\RR$ under the normalized excursion measure $\N_0^{(1)}$, we get that 
\begin{align} 
\limsup_{n\to\infty} 
\big( n^2 \times \eqref{Tc3terms0}_2\big)   
  &\le \frac{1}{\sigma\sqrt{2\pi}}\,\int_0^\infty
    s^{-3/2}\,  \N^{(1)}_{0}\pars*{\dmin(s^{1/4} \RR,  \coeffm^{-1}( K-x ))
    \le \varepsilon  } \rd s 
\nonumber\\
&=
 \frac{1}{\sigma}\,  \N_0(\dmin(  \RR, \coeffm^{-1}( K-x )) \le \varepsilon)
 \label{eq:int=Nx}
 \\
 &=
 \frac{1}{\sigma}\,  \N_{\coeffm^{-1}x}(\RR \cap
   (\coeffm^{-1}K)^{\varepsilon} \neq \emptyset),  \nonumber
\end{align}
where \eqref{eq:int=Nx} is due to~\eqref{eq:scaling-ISE}. This completes the proof of \eqref{inter-nK-up1}, and then of \eqref{upper:nK}. 

\medskip
{\bf Proof of  \eqref{lower:Open}.} The proof proceeds in a similar way to that of \eqref{upper:nK}. We are going to show that for any $K$ compact set and $x\neq0$ satisfying $K\subset \ball(0, |x|/2)$   and any $\varepsilon>0$, 
\begin{equation} 
\liminf_{n\rightarrow\infty}n^2\, \P_{\lfloor nx\rfloor}\pars*{\brwrange_c \cap n\cnbd K{2\varepsilon}\neq \emptyset}
\ge \frac{1}{\sigma}  \N_{\coeffm^{-1}x}\Big(\RR \cap   (\coeffm^{-1}K)  \neq \emptyset\Big). \label{inter-nK-low1}
\end{equation}
Indeed, given \eqref{inter-nK-low1}, we let $x \to\infty$ and successively use \eqref{eq:doubleeq}, \eqref{c_g} and Theorem \ref{p:snakecap} to obtain that
\begin{equation}\label{eq:infcompact}  
\liminf_{n\rightarrow\infty}\frac{\Bcap(n\cnbd K{2\varepsilon})}{n^{d-4}} \ge \frac {2}
{\sigma^2{c_{g}}}\, \BScap(\brwm^{-1/2} K). 
\end{equation}
Let $\open \subset\r^d$ be open. For any compact $K\subset \open$, there
exists       $\varepsilon=\varepsilon(\open,K)>0$        such       that
$\cnbd   K{2\varepsilon}\subset   \open$.    Then,   the   lower   bound
\eqref{eq:infcompact} implies that
\begin{equation*}
\liminf_{n\rightarrow\infty}\frac{\Bcap(n \open)}{n^{d-4}}  \ge \frac {2}
{\sigma^2{c_{g}}}\, \BScap(\brwm^{-1/2} K).
\end{equation*}
Since $\BScap$ is a
Choquet capacity and since the Borel sets are capacitable, we
obtain~\eqref{lower:Open} as a consequence of~\eqref{eq:capacitable}.

It remains to prove \eqref{inter-nK-low1}. Let $\delta\in (0, 1/2)$. 
By \eqref{eq:ctree}, \eqref{Tc3terms0} and \eqref{lower:inter-2}, we
deduce  that
\begin{multline*}
 \P_{\lfloor nx\rfloor}\pars*{\brwrange_c \cap n\cnbd K{2\varepsilon}\neq
   \emptyset}\\
 \begin{aligned}
& \ge \pars*{\frac {1}{\sigma\sqrt{2\pi}}+o(1)} \, \sum_{\delta n^4 \le
  j \le \delta^{-1} n^4} j^{-3/2}
\P\pars*{\brwrange_{(j)}\cap (nK^{2\varepsilon}-\lfloor nx\rfloor) \neq \emptyset }
\\
&\ge
\frac {1}{\sigma\sqrt{2\pi}} \, \sum_{\delta n^4 \le j \le \delta^{-1}
  n^4} j^{-3/2} \P\pars*{\dmin(n^{-1}j^{1/4} \RR_*,
  \coeffm^{-1}(K-x) ) \le \varepsilon }  
+ o(n^{-2}) \\
&\ge
\frac {1}{\sigma\sqrt{2\pi}} \, \int_{\delta n^4+1}^{\delta^{-1} n^4}
t^{-3/2} \P\pars*{\dmin(n^{-1}t^{1/4} \RR_*,     \coeffm^{-1}(K-x) )
  \le \varepsilon }  \, \rd t  
+ o(n^{-2}) ,
\end{aligned}
\end{multline*}
 where for the last inequality we have used the monotonicity on $t$ of the probability term. Using the change of variables $t= n^4 s$ and that $\RR_*$ under $\P$
is distributed as $\RR$ under the normalized excursion measure $\N_0^{(1)}$,
we get that
\begin{multline*}
\liminf_{n\rightarrow\infty}n^2\, \P_{\lfloor
    nx\rfloor}\pars*{\brwrange_c
    \cap n\cnbd K{2\varepsilon}\neq \emptyset}
  \\
  \ge  \frac {1}{\sigma\sqrt{2\pi}}
  \int_\delta^{\delta^{-1}} s^{-3/2}\,  \N^{(1)}_{0}\pars*{\dmin (s^{1/4}
    \RR,  \coeffm^{-1}( K-x )) \le \varepsilon  } \, \rd s. 
\end{multline*}
We let $\delta$ goes to $0$ and apply the monotone convergence theorem
and then \eqref{eq:int=Nx}. Then we get \eqref{inter-nK-low1} which
completes the proof of  \eqref{lower:Open}. 
\end{proof}

\section{Hitting probabilities and the branching capacity: Proof of Theorem \ref{theo:compareBcap}} \label{s:appendix}

\subsection{Preliminaries}\label{sec:tadj} The material of this subsection is derived from Zhu \cite{zhu2016critical}. 
We  introduce an  adjoint  Galton-Watson  tree which  is  a random  tree
derived from the critical Galton-Watson  tree $\ttree{c}$, with the only
modification   being   made  at   the   root.   In  the   adjoint   tree
$\ttree{\text{adj}}$    the    root     has    offspring    distribution
$\widetilde\mu=(\widetilde\mu(k))_{k\ge  0}$  instead   of  $\mu$,  with
$$\widetilde{\mu}(k)=\sum_{j=k+1}^\infty   \mu(j),$$   while   all   other
vertices retain the original offspring distribution $\mu$.

We then construct  an infinite tree $\ttree \I$ as  follows. We begin by
constructing             a              semi-infinite             branch
$\{\varnothing_0,  \varnothing_1, ...,  \varnothing_n, ...\}$  rooted at
$\varnothing_0$,  then we  root  a tree  $\ttree{\text{adj}}^i$ on  each
$\varnothing_i$, $i\ge 0$,  where $(\ttree{\text{adj}}^i)_{i\geq 0}$ are
independent trees distributed as  $\ttree{\text{adj}}$.  To every finite
or infinite tree $\ttree  \alpha$, with $\alpha\in\{\text{adj},\I\}$, we
define the  random walk $\tbrw  \alpha{}$ indexed by $\ttree  \alpha$ in
the  same way  as the  random  walk $\tbrw  c{}$ is  indexed by  $\ttree
c$. Under $\P_x$ those BRW are started from $x\in \z^d$.
Furthermore, we consider the sub-tree
$\ttree -$ of $\ttree \I$ with root $\varnothing_1$ defined by
$\ttree -= \ttree \I \backslash \ttree{\text{adj}}^0$, as shown in Fig.~\ref{fig:sec4.1}. 
We also   consider the restriction of the  BRW $\tbrw I{}$ to $\ttree -$
which we denote by  $\tbrw -{}$.

We write  $\brwrange_\alpha$ the range of $\tbrw \alpha{}$  for $\alpha\in\{\text{adj},\I,-\}$;   this is consistent with the notation  $\brwrange_c$.

Recall $d\ge 5$. 
Let   $K\subset\z^d$  be   a   finite  and   nonempty   set.  
We shall consider the hitting probabilities from $x\in \z^d$ by
\begin{equation}\label{def-rK}
\begin{aligned}
\phit \alpha (x):=\mathbf P_x(\brwrange_\alpha\cap
K\ne\emptyset)\quad\text{for}\quad
\alpha\in\{c,I,\text{adj},-\}.
\end{aligned}
\end{equation}
Notice that
under $\P_x$, the BRW $\tbrw - {}$ is distributed as $\tbrw \I {}$
started at $S_1$ and thus
\begin{equation}
  \label{eq:pK-pKI}
\phit - (x)=\E_x[\phit \I (S_1)], 
\end{equation}
where we recall that $S_1$ is one step of a random walk, distributed as $\theta.$
We shall consider the 
the  escape probabilities of $K$ by  $V_-$ 
\[
  {\bf e}_K(x):=1-\phit - (x).
\]
We deduce from Benjamini and Curien \cite[Theorem~0.1]{BC-12} that ${\bf e}_K$ is not
identically zero on $K$.
It is also shown in Zhu \cite[Eq.~(1.1)]{zhu2016critical} that
\begin{equation}
  \label{Bcap-eK}
\Bcap(K)= \sum_{a\in K} {\bf e}_K(a). 
\end{equation}
Using the monotonicity and the translation
invariance of $\Bcap$, we deduce that for any $a\in K$,
\begin{equation}
  \label{bcap>0}
    \Bcap(K) \ge \Bcap(\{a\})= \Bcap(\{0\})>0. 
\end{equation} 

We also recall that there exists some positive constant $C$,  only depending on $d$,  such that for all $r\ge 1 $ and  $K \subset \ball(0, r)$, we have \begin{equation} \Bcap(K) \le \Bcap(\ball(0,r))\le C\,  r^{d-4}. \label{upp-bcapK}    \end{equation}

\begin{figure}
\includegraphics[height=3.3cm]{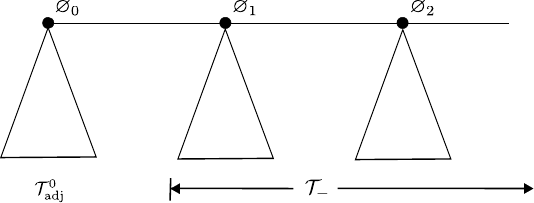}
\caption{\scriptsize
An illustration of $\ttree \I =\ttree
-\cup\ttree{\text{adj}}^0$.}
\label{fig:sec4.1}
\end{figure}

We remark that  for any $x\not\in K$,
\begin{equation}
\phit I (x) = \phit \adj (x) + (1- \phit \adj (x)) \phit
-(x).  \label{qrx}
\end{equation}

Moreover, these hitting probabilities are strongly related to a random walk $(S^\kappa)$ with killing rate $\phit\adj$. Denote by $\partial$ a cemetery point, and let $(S^\kappa)$ be a random walk on $\z^d \cup\{\partial\}$ such that for every $x\in \z^d$,  
\begin{equation}\label{eq:RW_with_killing}
\pp(S^\kappa_{n+1}=y | S^\kappa_n=x)= 
\begin{cases}    \phit\adj (x) , & \mbox{ if $y=\partial$},
\\
(1- \phit \adj (x))\theta(y-x) , & \mbox{ if $y\in \z^d$}.
\end{cases}
\end{equation}
The Green  function  of $S^\kappa$ is defined by
\begin{equation}  
G_K(x,y)
:=\sum_{n=0}^\infty \pp_x  (S^\kappa_n=y)
= \sum_{n=0}^\infty \ee_x \Big[{\bf 1}_{\{ S_n=y\}}\prod_{i=0}^{n-1}(1-\phit\adj{(S_i)})\Big].
\label{def-GK} \end{equation}

By \eqref{qrx}, we get that \begin{equation}    
\phit I(x)=\sum_{n=0}^\infty \ee_x \Big[ \phit\adj{(S_n)}\, \prod_{i=0}^{n-1}
(1-\phit\adj(S_i)) \Big] ,
\label{eq:q=gr} \end{equation}
with the convention $\prod_\emptyset:=1$. It is then immediate from
\eqref{eq:q=gr} and~\eqref{def-GK}
that
\begin{eqnarray} 
\phit I (x)= \pp_x((S^\kappa)\text{ survives})=\sum_{y\in\z^d}G_K(x,y) \phit\adj (y).  \label{eq:q=gr1} 
\end{eqnarray}
Moreover, an important relationship between $\phit c$ and $G_K$ has been obtained in Zhu \cite[Eq.~(5.4) and~(5.5)]{zhu2016critical}:
\begin{equation} 
\phit c(x)=\sum_{a\in K} G_K(x, a)= \sum_{n=0}^\infty
\pp_x(S^\kappa_n\in K)
\quad\text{for}\quad
x\in \z^d. \label{def:pKx}
\end{equation}

We summarize  some comparisons between the  probabilities $\phit \alpha (x)$ for $
\alpha\in\{c,I,\text{adj},-\}$. 
For $x\in K$, we note that $\phit \alpha (x)=1$ for  $
\alpha\in\{c,I,\text{adj}\}$, whereas $\phit -(x)\leq 1$.

\begin{lemma} \label{l:compar-palpha} Let $d\ge 5$ and $K\subset\z^d$ be
  a finite nonempty set.
    Under \eqref{hyp-tree} and    \eqref{hyp-brw},  
 for any $x\in \z^d\backslash K$, we have
\begin{align}
  \label{compar-pKr_K} 
    \frac{2( 1-\mu(0))}{\sigma^2}    \, \phit \adj (x)
    &\le   \phit c (x)
      \le \frac1{\mu(0)} \, \phit \adj (x),
    \\
     \label{eq:p->padj-pcpI}
   \phit c (x) \leq  \phit \I (x)
   &\quad\text{and}\quad
       \phit \adj (x)
       \leq  \frac{\sigma^2}{2} \,  \phit - (x),\\
\label{comp:pK-pKI}
\phit -(x) &\le    \phit I(x)
    \le \left(\frac{\sigma^2}{2}+1\right) \phit -(x).
 \end{align}
\end{lemma}

\begin{proof}   The   claim   \eqref{compar-pKr_K}  follows   from   the
  construction      of     $\tbrw      {\text{adj}}{}$,     see      Zhu
  \cite[Eq.~(8.6)]{zhu2016critical}.       The     first      inequality
  of~\eqref{eq:p->padj-pcpI}   is   a  consequence   of~\eqref{eq:q=gr1}
  and~\eqref{def:pKx}.   We  deduce   then  from~\eqref{eq:pK-pKI}  that
  $\E_x[\phit  c (S_1)]\leq  \E_x[\phit \I  (S_1)]= \phit  - (x)  $.  To
  obtain   the   second   inequality   of~\eqref{eq:p->padj-pcpI},   use
  \cite[Eq.~(8.4)]{zhu2016critical} to  get for $x\not\in K$  that, with
  $k_\varnothing$   the   number   of   children   for   the   root   of
  $\ttree{\text{adj}}$  and, for $i\in \{1, \ldots, k_\varnothing\}$,    $\brwrange^i$  the   range  of   the  BRW
  $V_\text{adj}$ restricted to the descendants of the $i$-th child of the
  root,
\[
  \phit \adj (x)
= \P_x (\brwrange^i \cap K \neq \emptyset \quad\text{for $i\leq
    k_\varnothing$})
 \leq  \sum_{i\geq 0} i \widetilde \mu(i) \,  \E_x[\phit c (S_1)]
 \leq   \frac{\sigma^2}{2}\, \phit - (x) .
\]
Lastly,  the  inequalities~\eqref{comp:pK-pKI} are a direct consequence
of~\eqref{qrx} and~\eqref{eq:p->padj-pcpI}. 
\end{proof}

We shall consider the first {exit} times from $B\subset \z^d$ by  the
$\theta$-random walk $(S_n)$  and 
$(S^\kappa_n)$
\begin{equation}
   \label{eq:def-tau}
   \tau_{S}(B):=\inf\{n\ge 1\, \colon\,  S_n \not\in B\}
   \quad\text{and}\quad
   \tau_{S^\kappa}(B):=\inf\{n\ge 1\, \colon\,   S^\kappa_n \not\in B\}.
\end{equation}
Then, we  define the harmonic measure of $(S^\kappa)$ with respect to every
nonempty set $B\subset\z^d$ by, for $x, y\in \z^d$
\begin{align}
H^B_K(x,y)&:=
\sum_{n=0}^\infty \pp_x  \big(S^\kappa_n=y, \tau_{S^\kappa}(B) \ge n)     \label{def-HBK0}
 \\
 &=   \sum_{n=0}^\infty \ee_x \Big[\prod_{i=0}^{n-1}(1-\phit\adj(S_i)),
     S_n=y, \tau_{S}(B) \ge n\Big].  \label{def-HBK}
 \end{align}
We note that neither $x$ nor $y$ needs to be in $B$.
We also note that for $x\in B$ and $ y\not\in B$
\begin{equation}
   \label{eq:HxBBc}
    H^B_K(x,y) =\pp_x(S^\kappa_ { \tau_{S^\kappa}(B)}=y)
= \ee_x \Big[\prod_{i=0}^{\tau_S(B)-1}(1-\phit\adj(S_i)),
     S_{\tau_{S}(B)}=y\Big].
  \end{equation}

By the  Markov property  of $S^\kappa$, we  easily obtain  the following
first  entrance  and  the   last  exit  decomposition  (see  \cite[Lemma
2.1]{zhu2016critical}): for any $x\in B, y\not\in B$,
\begin{align}
  \label{exit1}
  G_K(x,y)
  &=\sum_{z\not\in B}H^B_K(x,z)G_K(z,y)=\sum_{z\in
    B}G_K(x,z)H_K^{B^c}(z,y),\\
  \label{exit2}
  G_K(y,x)
  &=\sum_{z\in B}H^{B^c}_K(y, z)G_K(z, x)=\sum_{z\not\in B}G_K(y, z)H_K^B(z, x).
\end{align}

\subsection{Hitting probabilities and branching capacity}
We provide a comparison result between the  hitting probabilities of a finite
set $K$ by the branching random walks and the branching capacity.

The   following    result   is   given   in    \cite[Theorem   1.3   and
Eq.~(8.6)]{zhu2016critical},
establishing 

\begin{proposition}\label{prop:4.1}
  Let $d\ge 5$, $\lambda>1$. Under \eqref{hyp-tree} and \eqref{hyp-brw},
  uniformly in  $r\ge 1$,  $K \subset \z^d\cap  \ball(0,r) $  nonempty ,
  $x\in \z^d$ with $|x|\ge \lambda r$, we have
  \begin{align}
    \label{eq:killing_upperbound1}
\phit c(x)\asymp \phit \adj(x)  & \asymp |x|^{2-d} {\Bcap(K)},\\
    \label{eq:killing_upperbound3}
  \phit -(x) \asymp \phit I(x) &\asymp |x|^{4-d}{\Bcap(K)}.
\end{align}
\end{proposition}
Equation~\eqref{eq:killing_upperbound1} is   given   in   Zhu \cite[Theorem   1.3   and
Eq.~(8.6)]{zhu2016critical} and
Equation~\eqref{eq:killing_upperbound3} is also given in
Zhu \cite{zhu-2} when the displacements given by  $\theta$ are bounded. 
Notice that  the   first  part  of~\eqref{eq:killing_upperbound1}
and~\eqref{eq:killing_upperbound3} are   given
by~\eqref{compar-pKr_K} and~\eqref{comp:pK-pKI}
for  $x \not\in K$.

So    we   are    left   to    the    proof   of    the   second    part
of~\eqref{eq:killing_upperbound3}.      Under    the     general    case
\eqref{hyp-brw},     one      direction     of     the      proof     of
\eqref{eq:killing_upperbound3} is easy.

\begin{lemma}\label{lem:q=gp}
Let $d\ge 5$. Under \eqref{hyp-tree} and \eqref{hyp-brw}, 
uniformly in   $K\subset \z^d$ finite and nonempty,  and  $x \in \z^d$,  we have
\begin{equation}    
 \phit I(x)\asymp \sum_{y\in\z^d}G_K(x,y) \phit c(y). \label{qKx-1}
\end{equation}
Moreover, for any $\lambda>1$,  uniformly in $r\ge 1$,
$K\subset\ball(0,r)$ nonempty  and  $|x|\ge \lambda r$, we have
\begin{equation}
   \label{qKx:lower}
    \phit I(x)
   \gtrsim |x|^{4-d}\Bcap(K).
\end{equation}
\end{lemma}

\begin{proof} First, \eqref{qKx-1} follows immediately from
  \eqref{eq:q=gr1} and~\eqref{compar-pKr_K}.
By \cite[Lemma 12.3]{zhu2016critical},   uniformly in
$K\subset\ball(0,r)$ and $|x|, |y|\ge \lambda r$, we get, as
$|x|_\theta \asymp |x|$, that
\begin{equation}
  G_K(x, y) \asymp |x-y|^{2-d}. \label{GK-largexy}
\end{equation}
Then by \eqref{eq:q=gr1} and \eqref{eq:killing_upperbound1}, we obtain that
when $|x|\ge \lambda r$,
\begin{align}
\phit I(x) \ge &\sum_{2|x|\le |y|\le 3|x|}G_K(x,y)\phit\adj(y) \nonumber\\
\asymp&\sum_{2|x|\le |y|\le 3|x|}|x-y|^{2-d}\,|y|^{2-d}\, \Bcap(K) \nonumber
\\ 
\asymp& \, |x|^{4-d}\Bcap(K), \label{pKI>}
\end{align} 
proving \eqref{qKx:lower}.
\end{proof}

The other  direction of the proof of \eqref{eq:killing_upperbound3}  will be presented at the end  of this section after  some preliminary
results.  We  first cite a result  from \cite{zhu-2}, and note  that the
proof provided there, under the condition that the displacements given
by $\theta$ are bounded,  can be adapted to our case word by
word. 
 
\begin{lemma}{\cite[Eq.~(4.1)]{zhu-2}}\label{lem:ZhuII4.1}
Let $d\ge 5$. Under \eqref{hyp-tree} and \eqref{hyp-brw}, uniformly in $K\subset \z^d$ finite and nonempty,  $x\in \z^d$ and $B\subset \z^d$,
\[
\sum_{y\in B}G_K(x,y)\phit I(y)\lesssim (\diam(B)+1)^2\,  \phit I(x).
\]
\end{lemma}

We set
\[
  L^\kappa_n(B):= \sum_{i=0}^n \mathbf 1_{\{S^\kappa_i\in
    B\}}.
  \label{def-L} 
\]
By the Markov property, we get  that for any $B\subset \z^d$,
\begin{eqnarray}   
  \sum_{y\in B}G_K(x,y) \phit c(y)
  &=& 
\sum_{y\in B} \sum_{i=0}^\infty \pp_x(S^\kappa_i=y) \sum_{j=0}^\infty
      \pp_y(S^\kappa_j\in K) \nonumber 
\\
&=& \sum_{n=0}^\infty   \ee_x\left[L^\kappa_n(B) \, {\bf 1}_{\{
    S^\kappa_n\in K\}}\right].
    \label{eq:pg_to_gamma2}
\end{eqnarray}
In particular, for  $B=\z^d$ in \eqref{eq:pg_to_gamma2},  we obtain
\begin{equation}  \sum_{y\in \z^d }G_K(x,y) \phit c (y)
=\sum_{n=0}^\infty (n+1) \pp_x\Big(S^\kappa_n\in K\Big).  \label{eq:pg_to_gamma3} \end{equation}

\begin{lemma}
  \label{lem:pg_to_gamma4}
Let $d\ge 5$. Under \eqref{hyp-tree} and \eqref{hyp-brw}, uniformly in $K\subset \z^d$ finite and nonempty,  $x\in \z^d$ and $B\subset \z^d$,
\begin{equation}
  \sum_{y\in B} G_K(x,y) \phit I(y) 
  \gtrsim \sum_{n=0}^\infty
  \ee_x\Big[ L^\kappa_n(B)^2 \mathbf 1_{\{ S^\kappa_n\in K\}}\Big].
  \label{eq:pg_to_gamma4}
\end{equation}
  \end{lemma}
  
  \begin{proof}   
Using successively \eqref{qKx-1}, \eqref{def:pKx} and the Markov
property, we get that
\begin{eqnarray}   
 \sum_{y\in B} G_K(x,y) \phit I(y)  
 &\asymp& 
\sum_{y\in B}G_K(x,y) \sum_{z\in\z^d} G_K(y,z) \phit c(z)\nonumber
\\
&=&
\sum_{y\in B}G_K(x,y) \sum_{z\in\z^d} G_K(y,z) \sum_{\ell=0}^\infty \pp_z(S^\kappa_\ell\in K) \nonumber
\\
&=&  \sum_{y\in B} \sum_{i, j, \ell=0}^\infty \pp_x\Big(S^\kappa_i=y,  S^\kappa_{i+j+\ell}\in K\Big) \nonumber
\\
&=&
  \sum_{n=0}^\infty \sum_{i=0}^n (n-i+1) \pp_x\Big(S^\kappa_i \in B,  S^\kappa_n\in K\Big).  \label{eq:pg_to_gamma3-4} 
 \end{eqnarray}
 Note that
\[
  \sum_{i=0}^n (n-i+1) {\bf 1}_{\{S^\kappa_i \in B\}}
  = \sum_{j=0}^n\sum_{i=0}^j  {\bf 1}_{\{S^\kappa_i \in B\}}
\geq  \frac{1}{2}  \sum_{j=0}^n\sum_{i=0}^n
    {\bf 1}_{\{S^\kappa_j \in B\}} {\bf 1}_{\{S^\kappa_i \in B\}}
=  \frac{1}{2} L^\kappa_n(B)^2. 
\]
 
\noindent Then \eqref{eq:pg_to_gamma4} follows from \eqref{eq:pg_to_gamma3-4}. 
\end{proof}

 \begin{corollary}\label{lem:gp_small} 
 Let $d\geq 5$. Under \eqref{hyp-tree} and \eqref{hyp-brw}, uniformly in $K\subset \z^d$ finite and nonempty,  $x\in \z^d$ and $B\subset \z^d$,
 $$\sum_{y\in B}G_K(x,y) \phit c(y)
 \lesssim
 (1+\diam(B))  \, \sqrt{ \phit I(x)\, \phit c(x)}.
 $$ 
 \end{corollary}

\begin{proof}  By \eqref{eq:pg_to_gamma2} and the Cauchy-Schwarz inequality, \begin{eqnarray*}  
\sum_{y\in B}G_K(x,y) \phit c(y)
&=& \sum_{n=0}^\infty   \ee_x\Big[L^\kappa_n(B) \, {\bf 1}_{\{ S^\kappa_n\in K\}}\Big]
\\
&\le&
\sqrt{\sum_{n=0}^\infty   \ee_x\Big[L^\kappa_n(B)^2 \, {\bf 1}_{\{ S^\kappa_n\in K\}}\Big] } \sqrt{\sum_{n=0}^\infty   \P_x( S^\kappa_n\in K)}.
  \end{eqnarray*}  
  We    conclude    by    \eqref{def:pKx} and
  Lemmas~\ref{lem:ZhuII4.1} 
  and~\ref{lem:pg_to_gamma4}.
\end{proof}

\begin{lemma}\label{lem:ZhuII6.2}
Let $d\ge 5$, $\lambda>1$. Under \eqref{hyp-tree} and \eqref{hyp-brw},
uniformly in  $r\ge 1$, $K\subset\ball(0,r)$  and $|x|\ge \lambda r$, we
have
\begin{equation}   
\phit I(x)\lesssim |x|^{4-d}\Bcap(K). \label{qKx:upp}
 \end{equation}
\end{lemma}
\begin{proof}
  Let $B:=\ball(0,\lambda|x|)$. By  \eqref{eq:killing_upperbound1} and
  the fact that $G_K(x,y) \le g(x, y)$, we get
\begin{eqnarray}
\sum_{y\not\in B}G_K(x,y) \phit c(y) 
&\lesssim& \sum_{y\not\in B}|y|^{2-d}|y|^{2-d}\Bcap(K)  \nonumber
\\
&\asymp & |x|^{4-d}\Bcap(K). \label{eq:C1}
\end{eqnarray}
Then by \eqref{qKx-1} and  Corollay \ref{lem:gp_small}, we get that 
\begin{eqnarray*}
  \phit I(x)
  &\lesssim&  \sum_{y \in B}G_K(x,y) \phit c(y) + |x|^{4-d}\Bcap(K) 
\\
  &\lesssim&
             (1+\diam(B))  \, \sqrt{\phit I(x)\, \phit c(x)} +
           \sum_{y\not\in B}G_K(x,y) \phit c(y)  
\\
  &\lesssim&
             |x|\,\sqrt{ \phit I(x)\, \phit c(x)} + |x|^{4-d}\Bcap(K) .
\end{eqnarray*}
By \eqref{eq:killing_upperbound1}, $\phit c(x)\lesssim |x|^{2-d}\Bcap(K)$. 
The conclusion follows easily.
\end{proof}

\begin{proof}[Proof of Proposition~\ref{prop:4.1}.] Recall we only need
  to prove   the   second    part
of~\eqref{eq:killing_upperbound3}. 
Now, this  is a direct consequence of~\eqref{qKx:lower} and 
  Lemma~\ref{lem:ZhuII6.2}. 
  \end{proof}

\subsection{Proof of Theorem \ref{theo:compareBcap}}
In this section, we prove Theorem \ref{theo:compareBcap} following the the same ideas as \cite[Proposition 2.4, Lemma 2.5]{zhu-2} and \cite[Lemma 6.1, Lemma 7.1]{zhu2016critical}, while keeping track of the error terms at each step to obtain the exact asymptotic.  

We give below  an expression of the branching  capacity $\Bcap(K)$ using
$\phit -$  and the harmonic measure  $H_K^B$. In the proof,  we shall us
that  the  first  exit  time  $ \tau_S(B)$,  of  $B$  for  $(S_n)$,  see
by~\eqref{eq:def-tau}, is a.s. finite.

\begin{lemma}     \label{l:Bcap-eK}      Let     $d\geq      5$.     Let
  $K\subset B \subset \z^d$ be nonempty finite sets. We have
\begin{equation*}  
  \Bcap(K)= \sum_{a\in K} \sum_{b \not\in B} H^B_K(b, a) \,  {\bf e}_K (b).
\end{equation*}  
\end{lemma}

\begin{proof}  
Let $a\in K$. By construction, we have
\begin{eqnarray}     
  {\bf e}_K (a)
  &=&   \ee_a \Big[\prod_{i=1}^\infty (1- \phit\adj(S_i)) \Big] \nonumber
 \\
 &=& \sum_{b \not\in B} \ee_a \Big[\prod_{i=1}^{\tau_S(B)} (1-
     \phit\adj(S_i)) , S_{\tau_S(B)}=b \Big] \, {\bf e}_K (b) ,
     \label{eKab}
 \end{eqnarray} 
where we  used the  strong Markov  property of  $S$ for 
 the last equality.  For $b\not \in B$, we have
\begin{multline}
\ee_a \Big[\prod_{i=1}^{\tau_S(B)} (1- \phit\adj(S_i)) , S_{\tau_S(B)}=b
\Big]   \\
\begin{aligned}
   &= \sum_{n=1}^\infty \ee_a \Big[\prod_{i=1}^n (1- \phit\adj(S_i)) ,
   \tau_S(B)= n, S_n=b \Big] \\
  &=  \sum_{n=1}^\infty \ee_b \Big[\prod_{i=0}^{n-1} (1- \phit\adj(S_i))
  , \tau_S(B)\geq  n, S_n=a \Big]\\
  &=H^B_K(b, a),
\end{aligned}
 \end{multline}
 where we used the symmetry of $\theta$ for the second equality
 and~\eqref{def-HBK} for the last. 
Then use~\eqref{Bcap-eK} and~\eqref{eKab} to conclude. 
\end{proof}

\begin{lemma}\label{l:hbk} Let $d\geq 5$. Let $K\subset  B$ be two finite nonempty subsets of  $\z^d$ and $x\in B$.  We have
$$\sum_{a\not\in B}H^{B}_{K}(x,a) \ge  (1-\phit\adj(x))  {\bf e}_K (x). $$
\end{lemma}

\begin{proof}  
Recall the construction of $\ttree \I$ using the trees
$(\ttree{\text{adj}}^i)_{i\geq 0}$. We denote by  $\brwrange^{i}_{\text{adj}}$ 
 the range of the restriction of the  BRW $\tbrw I{}$ to $\ttree{\text{adj}}^i$. 
Thanks to~\eqref{qrx}, 
we have for $x\in B$
\begin{align*}
(1-\phit\adj(x))\,  {\bf e}_K (x)
=& 1-\phit I(x)\\
\le&\,
\P_x((\cup_{0\le i <\tau_S(B)}  \brwrange^{i}_{\text{adj}})\cap K=\emptyset) 
\\
=&
\sum_{a\not \in B} \P_x((\cup_{0\le i <\tau_S(B)}  \brwrange^{i}_{\text{adj}})\cap K=\emptyset, S_{\tau_S(B)}=a) 
\\
=& \sum_{a\not \in B} \ee_x \Big[ \prod_{i=0}^{\tau_S(B)-1} (1-\phit\adj(S_i)), S_{\tau_S(B)}=a  \Big]  
\\
=& \sum_{a\not\in B}H^{B}_{K}(x,a), \end{align*}

\noindent where the last equality follows from~\eqref{eq:HxBBc} and the
fact that $x\in B$. This ends the proof.
\end{proof}

\begin{figure}
\includegraphics[height=5cm]{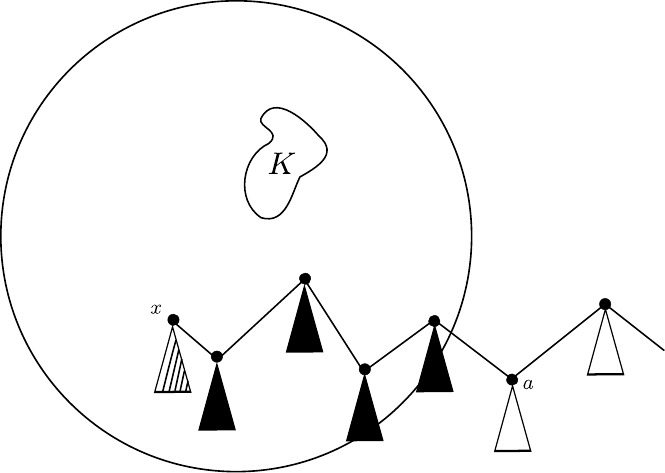}
\caption{\scriptsize
In this picture, we have a random walk starting at $x$ and exiting $B$ at the point $a$. On each point we attach an adjoint branching random walk, illustrated with the small triangles.
The quantity $\sum_{a\not\in B}H^{B}_{K}(x,a)$ is the probability that the shaded and solid triangles do not hit $K$,  ${\bf e}_K (x)=1-\phit-(x)=\P_x(\brwrange_-\cap K=\emptyset)$ is the probability that the solid and hollow triangles do not hit $K$,  $1-\phit\adj(x)=\P_x(\brwrange_{\text{adj}} \cap K=\emptyset)$ is the probability that the shaded triangle does not hit $K$, and  $1-\phit I(x)=\P_x(\brwrange_{\I}\cap K=\emptyset)$ is the probability that no triangle hits $K$.}
\label{fig:tree_xa}
\end{figure}

The main ingredient in the proof of Theorem \ref{theo:compareBcap} is the following estimate on the Green  function:

\begin{lemma}\label{lem:true_rs}
Let $d\ge 5$. Under \eqref{hyp-tree} and \eqref{hyp-brw}, there exists
some  constants $\alpha= \alpha(d), c=c(d)>0$ such that uniformly in $r\ge 1$,   $K \subset \ball(0,r)\cap\z^d$ nonempty, and   $x,y\in \z^d\backslash K$, we have 
\begin{align}
  \label{compareGkg}
  0\le 1-\frac{G_{K}(x,y)}{g(x,y)}\le c \, s^{-\alpha }
  \quad\text{with}\quad 
s=\frac{1}{r}\min\{|x|,|y|\}. 
\end{align}
\end{lemma}

\begin{proof} 
  By definition $G_K(x,y)\le g(x,y)$, see~\eqref{def-GK}, 
  so it is enough to show that the upper bound \eqref{compareGkg} holds for all $s\ge s_0$ where $s_0=s_0(d)>1$ denotes a large constant whose value only depends on $d$.

 By using \eqref{eq:killing_upperbound1} with
  $\lambda=s_0$ there and the fact $\Bcap(K) \lesssim r^{d-4}$ (see
  \eqref{upp-bcapK}), we get that $\phit\adj(x) \lesssim s^{2-d} r^{-2}
  \le s^{2-d}$;  the same holds for $\phit\adj(y)$. Since $\theta$ is symmetric, we have  
\begin{equation*}  \frac{G_K(x, y)}{1- \phit\adj(x)} = \frac{G_K(y,
    x)}{1- \phit\adj(y)}, 
\end{equation*}  
so  that uniformly  in  $r, K,  x,  y$  such that  $s\ge  s_0$, we  have
$G_K(y,   x)=  G_K(x,   y)   (1+   O(s^{2-d}))$.   Therefore,   provided
$\alpha\leq d-2$, it is enough to prove the lemma for $|x|\le |y|$.

So we assume that $|x|\le |y|$ and thus  $r s =|x|$.
Moreover, let $\upsilon=\upsilon(d)\in (0, \frac{d-4}{d})$
 be a small constant whose value will be fixed later, see Remark \ref{rmk:alpha1}. We distinguish two cases depending on  $|x|$ and $|y|$: 
\begin{align}\label{eq:1.1_0a}
|x| \le |y|\le |x| s^{\upsilon},
\end{align}
and
\begin{align}\label{eq:1.1_1a}
|y|> |x|  s^{\upsilon}.
\end{align}

{\bf Case  \eqref{eq:1.1_0a}: $|x|\le  |y| \le |x|  s^{\upsilon}$.}  See
the picture on  the left in Fig.~\ref{fig:prop4.8}. We  first modify the
Green   function   $g(x,y)$   of    the   random   walk   $S$   (without
killing).   For simplicity we write
\[
  B:=\ball(0, |x| s^{-\upsilon})\quad\text{and}\quad
  \tauB:=\tau_S(B^c)= \inf\{ n\geq 1 \, \colon\, S_n\in B\}
\]
the first
  exit time of $B^c$, see~\eqref{eq:def-tau}. 
Let     \begin{equation}      \widetilde     g(x,     y):=
  \sum_{n=0}^{s^\upsilon    (1+|x-y|)^2}     \pp_x(S_n=y,    \tauB    >
  n).
  \label{g-tilde}
\end{equation}
We claim that
\begin{equation}
  0\le  1- \frac{\widetilde g(x, y)}{g(x,y)} =O(s^{-\upsilon (d-2)/2}).
  \label{compar-g-tilde}
\end{equation}

  Indeed, we have
  \begin{eqnarray}  g(x,y)- \widetilde g(x, y)
    &\le& 
\sum_{n\ge s^\upsilon (1+|x-y|)^2} \pp_x(S_n=y) + \sum_{n=0}^\infty
          \pp_x(S_n=y, \tauB \le  n) \nonumber 
\\
    &=: &\eqref{g-tilde2cas}_1+ \eqref{g-tilde2cas}_2.
          \label{g-tilde2cas}
  \end{eqnarray}

By the local limit theorem for the random walk $S$ (see
\cite[Proposition 2.4.4]{lawler2010random}), we have $\P_x(S_n=y)\lesssim
n^{-d/2}$ uniformly in $x,y \in \z^d$. This gives that
\begin{equation}
  \eqref{g-tilde2cas}_1 \lesssim  s^{-\upsilon (d-2)/2} (1+|x-y|)^{2-d}
  \lesssim  s^{-\upsilon (d-2)/2} g(x,y).
  \label{g-tilde2cas-1}
\end{equation}

For $\eqref{g-tilde2cas}_2$, we use the Markov property at $\tauB$
to  get that
\begin{equation}
  \eqref{g-tilde2cas}_2
  = \sum_{a\in B}   \sum_{j=0}^\infty \pp_x(\tauB=j, S_j=a) g(a, y) 
  \le   \max_{a\in B} g(a, y) \, \pp_x(\tauB<\infty).
  \label{g-tilde2cas-2} \end{equation}

Since $g$ is the Green function,  $(g(S_n))_{n\in \n}$ is a supermartingale. Thus as
$x\neq 0$, we get
using~\eqref{c_g} that
\[
  g(x)\ge\ee_x[\mathbf 1_{\{\tauB<\infty\}}g(S_{\tauB})]
  \geq  \P_x( \tauB <\infty ) \, \min _{z\in B} g(z)
  \gtrsim \pp_x(\tauB<\infty)|x|^{2-d}s^{-\upsilon(2-d)},
\]
hence
\begin{equation}  \pp_x(\tauB<\infty) \lesssim s^{- \upsilon(d-2)}.  \end{equation}

Note                                                                that
$\min_{a\in  B}   |a-y|  \ge   |y|-  |x|   s^{-\upsilon}  \ge   |y|  (1-
s^{-\upsilon}),$                                                   hence
$ \min_{a\in B}\frac{|a-y|}{1+|x-y|}\ge  \frac 13, $ for  all $s\ge s_0$
and $s_0$  large enough  (depending only  on $\upsilon$).   This implies
that   $\max_{a\in    B}   g(a,   y)    \lesssim   g(x,   y),    $   and
therefore  \begin{equation} \eqref{g-tilde2cas}_2  \lesssim s^{-\upsilon
    (d-2)} g(x,y).\end{equation}

\noindent This and \eqref{g-tilde2cas-1} imply the claim in \eqref{compar-g-tilde}.

Now we go back to the proof of  the upper bound \eqref{compareGkg}. By \eqref{def-GK}, we have 
\begin{eqnarray}  G_K(x, y)
  &\ge&
\sum_{n=0}^{s^\upsilon (1+|x-y|)^2} \ee_x
        \Big[\prod_{i=0}^{n-1}(1-\phit\adj(S_i)), S_n=y, \tauB > n\Big].
        \nonumber
    \end{eqnarray}

    \noindent  By \eqref{upp-bcapK}  and \eqref{eq:killing_upperbound1},
    we                            deduce                            that
    $\phit\adj(S_i)\lesssim  r^{d-4} |S_i|^{2-d}  \le r^{d-4}  |x|^{2-d}
    s^{\upsilon(d-2)}$     for     any     $i    <     \tauB$.     Since
    $|x-y|\le  2|y|\le  2 |x|  s^{\upsilon}$,  there  are some  positive
    constants  $C,  C'$  (depending  only  on $d$)  such  that  for  all
    $n\le s^\upsilon (1+|x-y|)^2$ and $s\ge s_0$, on $\{\tauB > n\}$,
\[
   \prod_{i=0}^{n-1}(1-\phit\adj(S_i)) \ge (1- C r^{d-4} |x|^{2-d}
   s^{\upsilon(d-2)})^{s^\upsilon (1+|x-y|)^2} \ge 1- C' s^{- (d-4) +
     \upsilon (d+1)}.
 \]
Using \eqref{compar-g-tilde}, we see that  
\begin{equation}\label{eq:alpha1}
\begin{aligned}
G_K(x, y) &\ge  (1- C' s^{- (d-4) + \upsilon (d+1)})\,\widetilde g(x,y) 
\\
&\ge  (1-C'' s^{- (d-4) + \upsilon (d+1) }- C'' s^{-   \upsilon (d-2)/2}) g(x,y).
\end{aligned}
\end{equation}
This proves \eqref{compareGkg} in   the case \eqref{eq:1.1_0a} with
$\alpha $ at most equal to
\[
  \alpha_1:=\min (d-4 -v (d+1), v(d-2)/2).
\]

\begin{figure}
\includegraphics[height=5cm]{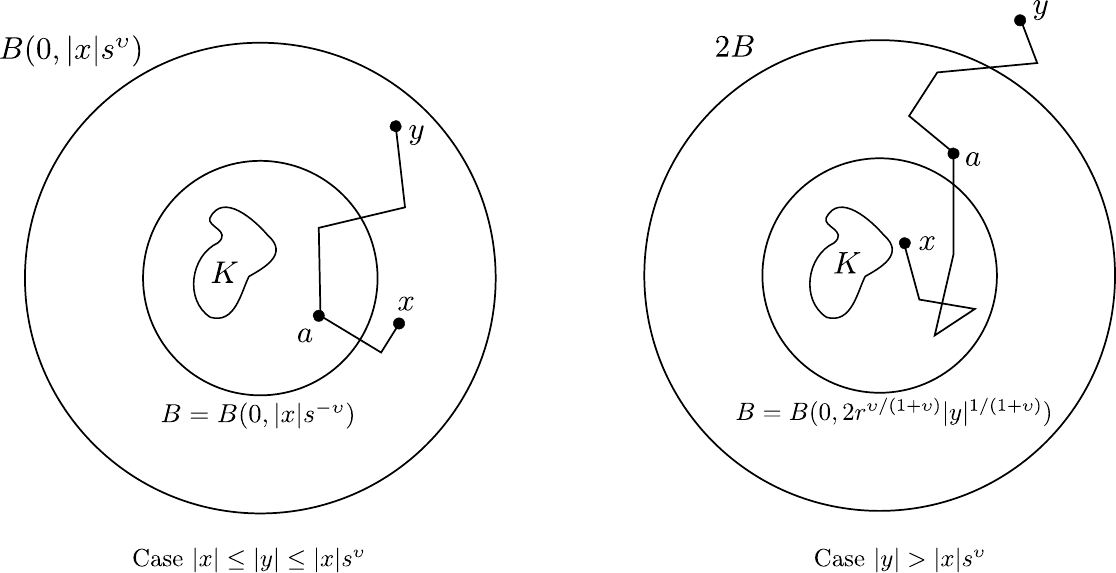}
\caption{\scriptsize
Relative positions of $a,x,y$ and $B$ in the proof of Lemma
\ref{lem:true_rs}. }
\label{fig:prop4.8}
\end{figure}

{\bf Case \eqref{eq:1.1_1a}: $|y|> |x| s^{\upsilon}$.} See the picture
on the right in Fig.~\ref{fig:prop4.8}.  We choose $R:= 2
r^{\upsilon/(1+\upsilon)}\, |y|^{1/(1+\upsilon)}$ and let $B:= \ball(0,
R)$, $2B=\ball(0, 2R)$. Note that $R \geq  2 r s=2|x|$. By taking $s\ge
s_0$ large enough (depending only on $v$), we have  for any
$a \in (2B)\backslash B$ that
\[
  |a| \le |y|\le |a| \pars*{\frac{|a|}{r}}^\upsilon.
\]
This means that  we may apply the already  proved case \eqref{eq:1.1_0a}
to $a$ and  $y$, and get from~\eqref{eq:alpha1} that  
\begin{equation*}
G_K(a, y) \ge (1- c' s^{-\alpha_1}) g(a,y) \ge  (1- c''
s^{-\alpha_1\wedge\upsilon}) g(x, y),  \qquad \forall a \in
(2B)\backslash B,
\end{equation*}

\noindent
where the last inequality follows from 
\eqref{c_g} and the fact that $\frac{|y|}{|x|} \ge 2\frac{|y|}{|a|}\ge
\frac{|y|}{R} \ge \frac1 2 s^{\upsilon}$
 so that  $|x-y|_\theta=|a-y|_\theta(1+ O(s^{-v}))$ and $|x-y|^{-1} +
|a-y|^{-1} \lesssim s^{-(v+1)}$.   Thus using \eqref{exit1}, we
get
\begin{equation}\label{eq:alpha2}
\begin{aligned}
G_{K}(x,y)
&=\sum_{a\not\in B}H^{B}_{K}(x,a)G_{K}(a,y)\\
&\ge \sum_{a\in (2B)\setminus B}H^{B}_{K}(x,a)G_{K}(a,y)\\
&\ge (1- c'' s^{-\alpha_1\wedge\upsilon}) g(x,y) \sum_{a\in (2B)\setminus B}H^{B}_{K}(x,a).
\end{aligned}
\end{equation}

The proof of \eqref{g-tilde2cas} will be complete once we show the existence of constants $\alpha_2 > 0$ and $c > 0$ such that  \begin{equation}    \sum_{a\in (2B)\setminus B}H^{B}_{K}(x,a) \ge 1- c s^{-\alpha_2}. \label{lower-HBK}\end{equation}
 
To get \eqref{lower-HBK}, we first deduce from Lemma \ref{l:hbk}  that
for all large $s\ge s_0$,
\begin{eqnarray}
  \sum_{a\not\in B}H^{B}_{K}(x,a)
  &\ge& (1-\phit\adj(x))  (1-\phit -(x))
\nonumber\\
&\ge&
 \pars*{1-C\frac{r^{d-4}}{(rs)^{d-2}}}\left({1-C\frac{r^{d-4}}{(rs)^{d-4}}}\right) 
\nonumber\\
&\ge& 1- C' s^{-(d-4)}, \label{sumnotB1}
\end{eqnarray}

\noindent where the second inequality follows from
\eqref{eq:killing_upperbound1}, \eqref{eq:killing_upperbound3}
and~\eqref{bcap>0}. Finally, by definition of $H^B_K$ in \eqref{def-HBK0} and
\eqref{hyp-brw}, we get
\begin{align*}    
  \sum_{a\not\in (2B)}H^{B}_{K}(x,a)
  &\le \sum_{a\not\in (2B)} \sum_{z\in B} \sum_{n=1}^\infty \P_x \Big(S^\kappa_n=a,  S^\kappa_{n-1} =z\Big)
  \\
   &\le
     \sum_{z\in
    B}\sum_{a\not\in(2B)}g(x,z)\theta(a-z) 
    \\
   &\lesssim R^{-(d-2)}\\
  &\le
  s^{-(d-2)}.
\end{align*}
Together  with  \eqref{sumnotB1},   this  imply  \eqref{lower-HBK}  with
$\alpha_2:=   d-4$   there.   This   completes   the   proof  of   Lemma
\ref{lem:true_rs}.
  \end{proof}

\begin{remark}\label{rmk:alpha1}
In the proof of Lemma \ref{lem:true_rs}, by \eqref{eq:alpha1}, \eqref{eq:alpha2} and \eqref{lower-HBK}, we need
\[
  \alpha\le
\min(\alpha_1, v, \alpha_2)=
  \min(d-4-\upsilon (d+1),\upsilon(d-2)/2,\upsilon,d-4).
\]
Hence the optimal choice for $\alpha$ in Lemma \ref{lem:true_rs} is
\[
\alpha=\upsilon=\frac{d-4}{d+2}\cdot
\]
\end{remark}

Now we are entitled to give the proof of Theorem  \ref{theo:compareBcap}.

\begin{proof}[Proof    of     Theorem    \ref{theo:compareBcap}]    Let
  $s:=    \frac{|x|}{r}\geq     \lambda$    and    fix     a    constant
  $\upsilon\in (0, \frac{d-4}{d})$.    By \eqref{eq:killing_upperbound1}, we may assume without loss of generality that $s\ge s_0$ with a large constant $s_0$ which only depends on $\lambda$ and $d$.     Let
  $R:= |x| s^{-\upsilon}=rs^{1-\upsilon}$ and $B:= \ball(0, R)$.
Notice that $ R <  |x|$ and $K \subset B$.
By \eqref{def:pKx} and \eqref{exit2}, we get
\begin{align}
  \P_x(\brwrange_c\cap K\neq \emptyset)
  =\phit c (x)
&=\sum_{a\in K} G_K(x, a) \nonumber
\\
&=\sum_{a\in K}\sum_{b\not\in B}G_K(x,b)H^B_K(b,a)
\nonumber
\\
&= (1+O(s^{-{(1-\upsilon)(d-4)}/{(d+2)}})) \sum_{a\in K}\sum_{b\not\in B} g(x,b)H^B_K(b,a), \label{sumKB1}
\end{align}
using  Lemma \ref{lem:true_rs} (with $s$
replaced by $s^{1-\upsilon}$) and Remark \ref{rmk:alpha1} for the last equality.

By \eqref{eq:killing_upperbound3}, uniformly in $b \not\in B$, 
\begin{align*}
\phit -(b) \lesssim \frac{\Bcap(K)}{R^{d-4}} \lesssim
  \Big(\frac{r}{R}\Big)^{d-4}
  = s^{-(1-\upsilon)(d-4)},
\end{align*}
which in view of Lemma \ref{l:Bcap-eK} yields that 
\begin{equation}  1 + O\big( s^{-(1-\upsilon)(d-4)}\big)=\frac{\Bcap(K)}{\sum_{a\in K}\sum_{b\not\in B}H^B_K(b,a)}\le 1 . \label{Bcap-HBK1}  
\end{equation}

Comparing \eqref{Bcap-HBK1}  with \eqref{sumKB1},  the proof  of Theorem
\ref{theo:compareBcap} reduces to  show that \begin{equation} \sum_{a\in
    K}\sum_{b\not\in  B} g(x,b)H^B_K(b,a)  = (1+  O(s^{-\alpha_3})) g(x)
  \sum_{a\in K}\sum_{b\not\in B} H^B_K(b,a), \label{gxb1} \end{equation}
for some constant $\alpha_3>0$.

To prove \eqref{gxb1}, we decompose both  sides of it into the sums over
$b \in  (2B)\backslash B$ and  $b\not\in (2B)$, with  $2B=\ball(0, 2R)$,
and  show that  those over  $b\not\in (2B)$  are negligible.  Indeed, by
\eqref{c_g},  uniformly  in  $x$,  $r$ and  $b  \in  (2B)\backslash  B$,
$g(x, b)= g(x) (1+ O(s^{-\upsilon}))$ as $\upsilon\leq 1$. Then, we have
\begin{equation}   \sum_{a\in K}\sum_{b\in (2B)\backslash B} g(x,b)H^B_K(b,a) =g(x) (1+ O(s^{-\upsilon}))\sum_{a\in K}\sum_{b\in (2B)\backslash B}  H^B_K(b,a).  \label{eq:b<2B}\end{equation}

Recall that for any $a\in K$,  $\Bcap(K)\ge
  \Bcap(\{a\})=\Bcap(\{0\})>0$.  
By \eqref{Bcap-HBK1} and \eqref{eq:b<2B}, \eqref{gxb1} follows once we establish the following two claims: 
\begin{align}
  \sum_{a\in K}\sum_{b\not\in (2B)}  H^B_K(b,a)
  & = O(s^{-\alpha_3})   ,
    \label{sumH>2B} \\
  \sum_{a\in K}\sum_{b\not\in (2B)} g(x,b)H^B_K(b,a)
  &= O(s^{-\alpha_3}) g(x)  .
    \label{sumgH>2B} 
\end{align}

\begin{figure}
\includegraphics[height=5cm]{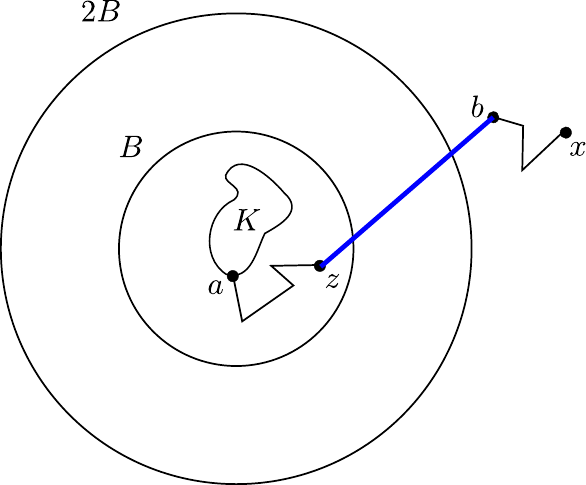}
\caption{\scriptsize
An illustration for the case $b\not\in (2B)$. Note that by definition of $ H^B_K(b,a)$ in \eqref{def-HBK0}, $S^\kappa_1$ makes a large jump (in blue) from  $b\not\in (2B)$ to some point $z\in B$,  so that $ H^B_K(b,a)\le \sum_{z\in B} \theta(z-b)G_K(z,a)$.}
\label{fig:prop1.1}
\end{figure}

Let $b\not\in (2B)$. To begin with, as illustrated in Fig.~\ref{fig:prop1.1}, we deduce from the definition of $H^B_K(b,a)$ in \eqref{def-HBK0} and  \eqref{def:pKx} that
  \begin{align}
    \sum_{a\in K}H^B_K(b,a)
    &\le \sum_{a\in K} \sum_{n=1}^\infty
      \P_b\Big(S^\kappa_n=a, S^\kappa_1\in B\Big) \nonumber
\\
    &=  \sum_{a\in K} \sum_{z\in B} \sum_{n=1}^\infty
      \P_b(S^\kappa_1=z) \, \P_z(S^\kappa_{n-1}=a) \nonumber
\\ 
&\le
\sum_{z\in B} \theta(z-b)\, \phit c(z). \label{eq:b_2B}
  \end{align}

\noindent This gives
\begin{equation}   \sum_{a\in K}\sum_{b\not\in (2B)}  H^B_K(b,a) 
\le
 \sum_{z\in B} \sum_{b\not\in (2B)}  \theta(z-b)\phit c(z).  \label{sumK2B-1}\end{equation}

We are going to show that 
\begin{equation} 
  \sum_{z\in B} \sum_{b\not\in (2B)}  \theta(z-b)\phit c(z)
  \le
  O(s^{- \alpha_4}),   \label{sumBpK22}  \end{equation}
  for some $\alpha_4\leq \alpha_3$, which in view of \eqref{sumK2B-1} yields \eqref{sumH>2B}.

 To show \eqref{sumBpK22}, we use  \eqref{hyp-brw} to see that uniformly in  $z\in B$,
\begin{align}\label{eq:d_moment}
\sum_{b\not\in (2B)} \theta(z-b)\lesssim R^{-d}.
\end{align}
It remains to show that
\begin{equation}
  R^{-d} \sum_{z\in B}    \phit c(z) = O(s^{-\alpha_4}).
  \label{sumBpK}
\end{equation}
To  this end,  we  decompose  $\sum_{z\in B}  $  in \eqref{sumBpK}  into
$\sum_{z\in              \ball(0,rs^{(1-\upsilon)/2})}$              and
$\sum_{z\in         B\setminus\ball(0,rs^{(1-\upsilon)/2})}$.        For
$|z| \ge rs^{(1-\upsilon)/2}$, we  use Proposition \ref{prop:4.1} to see
that    $\phit     c(z)    \lesssim    |z|^{2-d}     \Bcap(K)$.    Since
$K \subset \ball(0, r)$ and $\Bcap(\ball(0,r)) \lesssim r^{d-4}$, we get
that
\[
  \sum_{z\in B\setminus\ball(0,rs^{(1-\upsilon)/2})} \phit c(z)
\lesssim r^{d-4} \sum_{z\in B\setminus\ball(0,rs^{(1-\upsilon)/2})}
|z|^{2-d} \lesssim r^{d-4} R^2 = R^{d-2} s^{-(d-4)(1-\upsilon)}.
\]
For  the   sum  $\sum_{z\in  \ball(0,rs^{(1-\upsilon)/2})}$,   we  bound
$ \phit c(z)$ by $1$ and get that
\[ 
\sum_{z\in \ball(0,rs^{(1-\upsilon)/2})} \phit c(z) \lesssim r^d
s^{d(1-\upsilon)/2} = R^d s^{-d(1-\upsilon)/2},
\]
proving \eqref{sumBpK}, hence \eqref{sumBpK22}. The proof of
\eqref{sumH>2B} is   complete with $\alpha_3\leq  \alpha_4:=\min(d/2,
(d-4)(1-\upsilon))$. 

We now prove \eqref{sumgH>2B} in a similar way. By \eqref{eq:b_2B}, we
have
$$\sum_{a\in K}\sum_{b\not\in (2B)} g(x,b)H^B_K(b,a)
\le
\sum_{z\in B}\sum_{b\not\in (2B)}g(x,b)\theta(z-b)\phit c(z).
$$

\noindent Let $\eta>0$ be small (depending only on $d$). We decompose
the above sum $\sum_{b\not\in (2B)}$ into $\sum_{b\not\in (2B),  |b-x|\le
  |x| s^{-\eta}} $ and  $\sum_{b\not\in (2B),  |b-x|> |x| s^{-\eta}}
$. For $ |b-x|> |x| s^{-\eta}$, we have $g(x, b) \lesssim |x|^{2-d}
s^{\eta(d-2)} \lesssim g(x)s^{\eta(d-2)} $. This implies that
\begin{eqnarray}    
\sum_{z\in B}\sum_{\substack{b\not\in (2B)\\  |b-x|> |x| s^{-\eta}}}
  g(x,b)\theta(z-b)\phit c(z) 
&\lesssim&
g(x)s^{\eta(d-2)} \sum_{z\in B}\sum_{b\not\in (2B)}\theta(z-b)\phit c(z)\nonumber
\\
&\lesssim&
g(x)s^{\eta(d-2)- \alpha_4 } , \label{sumgH>2B1}
\end{eqnarray}

\noindent where the last inequality is due to \eqref{sumBpK22}.

For the sum $\sum_{b\not\in (2B),  |b-x|\le |x| s^{-\eta}} $, we bound
$\phit c(z)$ by $1$ and use $g(x, b) \lesssim (1+|x-b|)^{2-d}$. Then we
have
$$
  \sum_{z\in B} \sum_{\substack{b\not\in (2B)\\  |b-x|\leq   |x| s^{-\eta}}}
  g(x,b)\theta(z-b)\phit c(z) 
\lesssim   
\sum_{z\in B} \sum_{\substack{b\not\in (2B)\\  |b-x|\leq   |x|
    s^{-\eta}}}
(1+|x-b|)^{2-d}  \theta(z-b).
$$

\noindent Note that for any $ |b-x|\le |x| s^{-\eta}$ and $z\in B$, we
have $|z-b|\ge |x|/2$ (assuming $s_0^{-\upsilon} +s_0^{-\eta}\leq  1/2$), hence $\sum_{z\in B} \theta(z-b) \le {\mathbf
  P}_0(|S_1|\ge |x|/2) \lesssim |x|^{-d}$ by \eqref{hyp-brw}.  It
follows that
\begin{eqnarray*}
  \sum_{z\in B} \sum_{\substack{b\not\in
  (2B)\\  |b-x|\leq   |x| s^{-\eta}}}
  g(x,b)\theta(z-b)\phit c(z)
&\lesssim&
|x|^{-d} \sum_{|b-x|\le |x| s^{-\eta}} (1+|x-b|)^{2-d}  
\\
&\lesssim& 
|x|^{-d} \, (|x| s^{-\eta})^2
\\
&\lesssim& g(x) s^{-2 \eta}. \end{eqnarray*}

\noindent
Let $\eta:= \min(d/2, (d-4))\, (1-\upsilon)/d$. This combined with
\eqref{sumgH>2B1} yields \eqref{sumgH>2B} with
\begin{equation}\label{eq:alpha3}
\alpha_3:=\min (\alpha_4- \eta(d-2),  2\eta).
\end{equation}
This completes the proof of Theorem  \ref{theo:compareBcap}.
\end{proof}

\begin{remark}
By \eqref{sumKB1}, \eqref{eq:b<2B}, \eqref{sumBpK22} and \eqref{eq:alpha3},
we have
\[
  \alpha\le\min\pars*{
    \frac{(d-4)(1-\upsilon)}{d+2},
    (d-4)(1-\upsilon), 
    \upsilon,
    (1-\upsilon),
    \frac{2(d-4)(1-\upsilon)}{d}
 }
\]
for some $0<\upsilon<\frac{d-4}d$ (independent from the parameter
$\upsilon$ in Remark \ref{rmk:alpha1}). The optimal choice is 
\[
 \alpha=\upsilon=\frac{d-4}{2(d-1)}\cdot 
\]
\end{remark}

\subsection{Consequences of \texorpdfstring{Theorem  \ref{theo:compareBcap}}{}}
As a by-product, we can approximate $\Bcap(K)$ in terms of $\phit \alpha
(x)$ for $\alpha\in \{ \I, \adj, -\}$  with the same precision as in
Theorem \ref{theo:compareBcap}. These approximations hold  under an
additional third-moment assumption on the offspring distribution $\mu$. 

\begin{proposition}\label{c:compareBcap} Let $d\ge 5$ and
  $\lambda>1$. Assume \eqref{hyp-tree}, \eqref{hyp-brw} and that $\mu$
  has a finite third moment.   There exists a positive constant  $C=C(d,
  \lambda)$ such that, uniformly in $r\ge 1$, $K \subset
  \ball(0,r)\cap\z^d$ nonempty, $x\in \z^d$ such that
$|x|\ge \lambda r$,  we have, with $\alpha$ as in Theorem~\ref{theo:compareBcap},
\begin{equation*} 
\Big|\frac{ \phit\adj(x)}{g(x)} - \frac{\sigma^2}{2}\Bcap(K) \Big|  \le  C \pars*{\frac{r}{|x|}}^{\alpha}\Bcap(K).
\end{equation*}
\end{proposition}

\begin{proof}     Let    $s:=\frac{ |x|}{r}\geq    \lambda$.     By
  \eqref{eq:killing_upperbound1},   we  may   assume  without   loss  of
  generality that $s\ge  s_0$ with some large constant  $s_0$ which only
  depends on $d$ and $\lambda$.  Let $\widetilde X$ be distributed according to
  $\widetilde \mu$, that is, the number of children of the root in
  $\ttree {\text{adj}}$. Since $\mu$ has finite third moment we deduce
  that $\E[\widetilde X^2]\le \sum_{k\in \n} k^3 \mu(k)$ is finite. 
 By definition  of $\tbrw{\text{adj}}x$, we
have
\[
  \phit   \adj(x)   =   1-   \E\left[(1-   \E_x[\phit c(S_1)])^{\widetilde
      X}\right].
\]
Since  $1- \widetilde X t\leq  (1-t)^{\widetilde 
  X}\leq 1- \widetilde X t
+  \widetilde X(\widetilde X-1) t^2$  for $t\in [0, 1]$, we
deduce that
\begin{equation}    
0\le \E[\widetilde X] -\frac{\phit \adj(x)}{\E_x[\phit c(S_1)]} \le
\E[\widetilde X^2]
\, \E_x[\phit c(S_1)],
\qquad \forall \, x\in \z^d.
\label{rK-t(x)}
\end{equation} 
Since  $\E[\widetilde X]=\sigma^2/2$ and $g(x)=O(s^{2-d})$, the proof is complete once we
prove that for $|x|=rs$ and some $\alpha'<d-2$
\begin{equation}   
   \E_x[\phit c(S_1)]
   = g(x) \Bcap(K) \Big(1+ O\big(s^{-\alpha'}\big)\Big).
   \label{eq:t(x)}
 \end{equation}
We decompose the left hand side according to $|S_1-x|$ less or larger than $s^{-\upsilon} |x|$,
with $\upsilon\in (0,1)$. 
By \eqref{hyp-brw}, we get 
\begin{equation}
   \label{eq:a'-grand}
   \E_x[\phit c(S_1) 1_{\{|S_1-x|> s^{-\upsilon}|x|\}}]\leq   \P_x(|S_1-x|> s^{-\upsilon}|x|)
   \lesssim s^{\upsilon d } |x|^{-d}
    \lesssim   s^{\upsilon d - 2}  \, g(x).
 \end{equation} 
Under $\P_x$, on $\{|S_1-x|\le s^{-\upsilon} |x|\}$, we have $|S_1|\geq (1-
 s^{-\upsilon})|x|\geq \lambda' r$ for some $\lambda'>1$. Using  Theorem
 \ref{theo:compareBcap} and $\E_x[g(S_1)]=g(x)$ for $x\neq 0$, we get
\begin{multline*}
    \Big|\frac{\E_x[\phit c(S_1) 1_{\{|S_1-x|\leq s^{-\upsilon} |x|\}}]}{\Bcap(K)} -g(x) \Big|\\
\begin{aligned}
&=   \Big|\frac{\E_x[\phit c(S_1) 1_{\{|S_1-x|\leq s^{-\upsilon} |x|\}}]}{\Bcap(K)}
-\E_x[g(S_1) ] \Big|\\ 
  &\lesssim  \E_x[g(S_1)1_{\{|S_1-x|> s^{-\upsilon} |x|\}}] + 
    \E_x\big[g(S_1) \big(\frac{r}{|S_1|}\big)^\alpha 1_{\{|S_1-x|\leq
  s^{-\upsilon} |x|\}} \big ]\\
   &\lesssim   \P_x(|S_1-x|>s^{-\upsilon} |x|) +  s^{-\alpha} g(x)\\
  &\lesssim s^{-\alpha'}  g(x),
     \end{aligned}      
\end{multline*}
where we used~\eqref{eq:a'-grand}  and $\alpha'\leq \min (2-\upsilon d , \alpha)$  for the last
equality. Taking, $\upsilon$ such that $2 - \upsilon d= \alpha$, we
deduce that~\eqref{eq:t(x)} holds with $\alpha'=\alpha$.
\end{proof}

\begin{proposition}\label{c:compareBcap2}   Let $d\ge 5$ and
  $\lambda>1$. Assume \eqref{hyp-tree}, \eqref{hyp-brw} and that $\mu$
  has a finite third moment.   There exists a positive constant  $C=C(d,
  \lambda)$ such that, uniformly in $r\ge 1$, $K \subset
  \ball(0,r)\cap\z^d$ nonempty, $x\in \z^d$ such that
$|x|\ge \lambda r$,  we have, with $\alpha$ as in
Theorem~\ref{theo:compareBcap},

\begin{equation}
  \label{qK-G}
  \Big|\frac{\phit I(x)}{G(x)\Bcap(K)} - \frac{\sigma^2}{2} \Big|
  \le  C \pars*{\frac{r}{|x|}}^\alpha
  \quad\text{and}\quad
  \Big|\frac{\phit -(x)}{G(x)\Bcap(K)} - \frac{\sigma^2}{2} \Big|
  \le  C \pars*{\frac{r}{|x|}}^\alpha, 
\end{equation}
 where 
\[
G(x)=\sum_{y\in\z^d}g(x,y)g(y)\asymp (1+|x|)^{4-d}.
\]
\end{proposition}

We mention Schapira \cite[Lemma 2.5]{Schapira} for an estimate on
$\phit I(x)$ 
similar to \eqref{qK-G}, in the case when $\theta$ is  uniformly distributed among the $2d$ unit vectors in $\z^d$.

\begin{proof}  Let $s:= \frac{|x|}{r}\geq \lambda$. By \eqref{eq:killing_upperbound3}, we may assume without loss of generality that $s\ge s_0$ with some large constant $s_0$ which only depends on $d$ and  $\lambda$.  
Let $\lambda':= (1+\lambda)/2 >1$. By \eqref{eq:q=gr1}, we have
\begin{align*}
\phit I(x)
=&\sum_{y\in\z^d}G_K(x,y)\phit\adj(y)\\
=&\sum_{y\in\ball(0,\lambda' r)}G_K(x,y)\phit\adj(y)+\sum_{y\not\in\ball(0,\lambda' r)}G_K(x,y)\phit\adj(y).
\end{align*}

\noindent Using successivley \eqref{compar-pKr_K}, Corollary \ref{lem:gp_small}, \eqref{eq:killing_upperbound1} and \eqref{qKx:upp}, we have
\begin{eqnarray*}    \sum_{y\in\ball(0,\lambda' r)}G_K(x,y)\phit\adj(y)
&\lesssim &
\sum_{y\in\ball(0,\lambda' r)}G_K(x,y)\phit c(y)
\\
&\lesssim & r  \sqrt{\phit c(x) \phit I(x)}
\\
&\lesssim & r|x|^{3-d}\Bcap(K)=s^{-1}|x|^{4-d}\Bcap(K).
\end{eqnarray*}

For the sum $\sum_{y\not\in\ball(0,\lambda' r)}$, we deduce from Lemma \ref{lem:true_rs} and Proposition \ref{c:compareBcap} that
\begin{align*}
\sum_{y\not\in\ball(0,\lambda' r)}G_K(x,y)\phit\adj(y)
&=(1+O(s^{-\alpha}))\frac{\sigma^2}{2}\Bcap(K)\sum_{y\not\in \ball(0,r)}g(x,y)g(y)\\
&=(1+O(s^{-\alpha}+s^{-2}))\frac{\sigma^2}{2}\Bcap(K)G(x),
\end{align*}
where we used that
\[
G(x)-\sum_{y\not\in \ball(0,r)}g(x,y)g(y)
=\sum_{y\in \ball(0,r)}g(x,y)g(y)\lesssim |x|^{2-d}\sum_{y\in\ball(0,r)}g(y)\lesssim |x|^{4-d}s^{-2}.
\]
This proves the first part of~\eqref{qK-G}.

Finally, by \eqref{qrx} and the fact that $\frac{\phit\adj(x)}{\phit
  I(x)}\lesssim |x|^{-2}$, $\phit\adj(x) \lesssim g(x) \Bcap(K) \lesssim
|x|^{2-d} r^{d-4},$ we immediately conclude the second part
of~\eqref{qK-G} from  its first part. 
\end{proof}

{\bf Funding.}
Yueyun Hu was partially supported by ANR LOCAL.

\end{document}